\documentclass{article}

\usepackage{amsthm, amsmath, amssymb, amsrefs, mathrsfs, dsfont, tikz-cd, stmaryrd, enumitem, hyperref, graphicx}
\usepackage[utf8]{inputenc}
\usepackage[a4paper,margin=30mm]{geometry}
\usepackage{titlesec}

\theoremstyle{definition}
\newtheorem{definition}{Definition}[section]

\newtheorem{remark}[definition]{Remark}
\theoremstyle{plain}
\newtheorem{lemma}[definition]{Lemma}

\newtheorem{theorem}[definition]{Theorem}
\newtheorem{corollary}[definition]{Corollary}

\newtheorem{proposition}[definition]{Proposition}

\newtheorem{introtheorem}{Theorem} 	%  
\newtheorem{introcorollary}[introtheorem]{Corollary} 	%  
\newcommand{\R}{\mathbb{R}}
\newcommand{\N}{\mathbb{N}}
\newcommand{\Z}{\mathbb{Z}}

\renewcommand{\cal}[1]{\mathcal{#1}}

\newcommand{\qet}{\hfill $\triangle$}
\renewcommand{\ker}{\mathrm{ker}\:}

\newcommand{\T}{\mathcal{T}}
\newcommand{\A}{\mathbb{A}}
\renewcommand{\P}{\mathcal{P}}
\newcommand{\U}{\mathcal{U}}
\newcommand{\Mtilde}{\widetilde{M}}
\newcommand{\Stilde}{\widetilde{S}}
\newcommand{\Mbar}{\overline{M}}

\title{Rigidity of coarsely minimal Reeb flows}
\author{B. Bramham and J.S. de Pooter}
\date{\today}

\begin{document}

\maketitle

\begin{abstract}
     We introduce the notion of a coarsely minimal Reeb flow, generalizing the notion of minimal geodesic flow, and prove the following rigidity theorem: That a coarsely minimal Reeb flow satisfying a divergence property is orbitally equivalent to the geodesic flow of a Riemannian metric of negative sectional curvature. Without the divergence assumption, we obtain an orbital semi-equivalence. This extends a rigidity result for geodesic flows of negatively curved Riemannian metrics which is due to Gromov \cite{gromov2000three}.  We use Floer homology and Morse's hyperbolic `stability' Lemma \cite{morse1924fundamental}.
\end{abstract}

\tableofcontents

\section{Introduction}

\subsection{Description of main results}
Let $M$ be a smooth closed manifold that admits a background Riemannian metric $g_b$ of negative sectional curvature.    It is well known that the geodesic flow of any other Riemannian metric $g$ of negative sectional curvature on $M$ is orbit equivalent to the geodesic flow of $g_b$: This was shown by Gromov \cite{gromov2000three} building on Morse's foundational work on hyperbolic rigidity in 1924 \cite{morse1924fundamental}.   The geodesic flow of $g$ corresponds to the Reeb flow on the unit co-sphere bundle 
\[
				S^*_{g}M:=\big\{ p\in T^*M\ \big|\ \|p\|_{g_*}=1\big\},
\] 
where the latter is equipped with the contact form induced by the Liouville $1$-form on the cotangent bundle $T^*M$ and $g_*$ is the metric dual to $g$.    In this article we extend this rigidity result to more general Reeb flows, that is, to more general fiberwise star-shaped hypersurfaces $S\subset T^*M$.     

A natural question is whether the condition of negative curvature on $g$ can be replaced by a condition on Reeb trajectories.   Here we were guided by the work of Knieper \cite{knieper2002hyperbolic} who extended the result from geodesic flows of negative curvature to those which have no conjugate points and satisfy a divergence property we recall below.   The question is then whether these conditions admit suitable analogues for Reeb flows.  
 The divergence condition makes sense for a Reeb flow of a general fiberwise star-shaped hypersurface $S\subset T^*M$, but a satisfactory notion of ``having no conjugate points'' proved trickier to generalize.   For example the definition introduced in \cite{contreras2003asymptotic}, while natural, would force the fibers of $S$ to be convex.   This is a restriction we wished to avoid, if possible, as it would effectively mean that our results would apply ``only'' to Finsler geodesic flows - which would still be interesting, but not as interesting.   To get beyond this we started from the perspective that a geodesic flow has no conjugate points if and only if every geodesic is minimal when lifted to the universal covering $\Mtilde$.    Although the concept of minimality can be readily extended to Reeb flows (on a fiberwise star-shaped hypersurface $S\subset T^*M$) it again has the drawback of forcing $S$ to be fiberwise convex, .   However we found a weaker notion which we call \textit{coarse minimality}, and this notion does not force the fibers of $S$ to be convex, and is strong enough for our purposes.   In short, in Gromov's rigidity result above, we can replace the condition that $S$ is the unit co-sphere bundle of a negative curvature Riemannian metric $g$,  
by the condition that the Reeb flow on $S$ is coarsely minimal and satisfies the divergence property:    

\begin{introtheorem}\label{thm:main 1 intro}
Suppose $M$ is a closed manifold admitting a background Riemannian metric $g_b$ of negative sectional curvature.   If $S\subset T^*M$ is a fiberwise star-shaped hypersurface whose Reeb flow is 
coarsely minimal and divergent, then there exists an orbital equivalence between the Reeb flow $\phi^t_{S}$ and the geodesic flow $\phi^t_{g_b}$;    
that is, a homeomorphism $F:S\to S_{g_b}M$ sending orbits of $\phi^t_{S}$ to orbits of $\phi^t_{g_b}$:
\[
				\quad\qquad F\big(\phi^{\R}_{S}(p)\big)=\phi^{\R}_{g_b}\big(F(p)\big)\qquad \forall p\in S.
\]
\end{introtheorem}

This is the main result of the paper, see Theorem \ref{T:main} for a more precise statement.  

Let us explain the notion of coarse minimality: We say that a Reeb trajectory $\gamma:\mathbb{R}\to S\subset T^*M$ is $T$-\textit{minimal}, $T\geq 0$, if the following holds: 
For each time interval $[s,t]\subset\mathbb{R}$ the ends can be modified by an amount not greater than $T$ to $[s',t']\subset\mathbb{R}$, so that $\gamma$ restricted to $[s',t']$ has minimal action 
amongst all Reeb chords in the homotopy class represented by $\gamma_{[s',t']}$, where the end points are free to move within the end point fibers of $T^*M$.  
We say the Reeb flow of a fiberwise star-shaped hypersurface $S\subset T^*M$ is \textit{coarsely minimal} if there exists $T\geq 0$ so that every Reeb trajectory is $T$-minimal.   In Sections \ref{S:geodesics} and  \ref{S:coarse min} we recall the Riemannian phenomena and explain our definition for Reeb flows precisely.  

This terminology is motivated by the following:  If, as above, $S=S^*_gM$ is the unit co-sphere bundle of a Riemannian metric $g$ on $M$, then the Reeb flow on $S$ is conjugate via the Legendre transform to the geodesic flow of $g$ on $M$, and a geodesic for the latter is minimal if and only if the corresponding Reeb trajectory is $0$-minimal in our terminology above.   

We say that the Reeb flow of a fiberwise star-shaped hypersurface $S\subset T^*M$ is \textit{divergent} if for every pair of geometrically distinct Reeb trajectories, their projections 
down to $M$ have lifts to the universal cover $\tilde{\gamma}_0,\tilde{\gamma}_1:\R\to \Mtilde$ so that $\lim_{t\to\infty}d(\tilde{\gamma}_0(t),\tilde{\gamma}_1(t))=\infty$, equivalently the Hausdorff distance is infinite; $d_{H}(\tilde{\gamma}_0(\R),\tilde{\gamma}_1(\R))=\infty$.   The concept of divergence as applied to geodesic flows is explained in \cite{knieper2002hyperbolic}, see also references therein and \cites{climenhaga2021uniqueness,climenhaga2022closed,knieper2025uniqueness}.

Our second main result makes precise something that we mentioned above, namely that coarse minimality for Reeb flows is relaxed enough to permit non-convex fibers; while $0$-minimality forces the fibers of $S$ to be convex, see Proposition \ref{propn:minimal implies convex intro}, $T$-minimality for $T>0$ does not: 

\begin{introtheorem}\label{thm:example intro}
Suppose $M$ is a closed manifold admitting a background Riemannian metric $g_b$ of negative sectional curvature.  
For each $T>0$ there exists a fiberwise star-shaped hypersurface $S\subset T^*M$ for which the Reeb flow is $T$-coarsely minimal (and divergent), where $S$ is not fiberwise convex; that is, 
there exist $q\in M$ so that $S\cap T^*_qM$ does not bound a convex region.   
\end{introtheorem}

This shows that Theorem \ref{thm:main 1 intro} applies to a larger class of Reeb flows than Finsler geodesic flows - in particular it goes beyond negative curvature Riemannian geodesic flows.   Theorem \ref{thm:example intro} is proven in Section \ref{S:example}.   

We end by mentioning a corollary of our main result that uses Gromov again \cite{gromov2000three}: If $(M_1,g_1)$ and $(M_2,g_2)$ are closed Riemannian manifolds with negative curvature and isomorphic fundamental groups, then there exists an orbital equivalence between the geodesic flows $\phi^t_{g_1}$ and $\phi^t_{g_2}$.  In particular this implies the non-trivial topological statement that the unit tangent bundles $S_{g_1}M_1$ and $S_{g_2}M_2$ are homeomorphic - even though $M_1$ and $M_2$ need not be homeomorphic (or apriori even have the same dimension).  Combining this with Theorem \ref{thm:main 1 intro} gives: 

\begin{introcorollary}\label{thm:main 2 intro}
Let $M_1$ and $M_2$ be smooth closed manifolds that have isomorphic fundamental groups and support metrics of negative curvature.    Suppose $S_i\subset T^*M_i$ 
are fiberwise star-shaped hypersurfaces whose Reeb flows are coarsely minimal and divergent.    Then these Reeb flows are orbitally equivalent.   
\end{introcorollary}

\subsection{Outline of the proof}
Here we outline the main steps in the proof of Theorem \ref{thm:main 1 intro}.   

The first step is that if the Reeb flow on $S\subset T^*M$ satisfies the coarse minimality condition (the divergence property is not
needed for this step), where $M$ is a closed manifold admitting a background Riemannian metric $g_b$ of negative curvature, then every Reeb trajectory 
in $S$ projects under the bundle map $\pi:T^*M\to M$ to a quasi-geodesic in $M$  with respect to $g_b$.   This involves establishing two 
estimates, one of which follows simply from compactness of $M$ while the other uses the Lagrangian Floer homology for pairs of fibers in $T^*M$, associated to a 
$2$-homogenous Hamiltonian $H:T^*M\to\R$ for which $S$ is the energy surface $\{H\equiv 1/2\}$.    This step would also fit naturally within the framework of wrapped 
Floer theory, but since we do not need to use results about the Fukaya category we felt that this more classical set up is also suitable.  

The second step is to apply a Lemma of Morse, or as later reformulated in terms of quasi-geodesics, which says that each quasi-geodesic in the universal cover $\widetilde{M}$ coarsely shadows a genuine geodesic for the negative curvature metric $g_b$ \cite{morse1924fundamental}, or see \cite{knieper2002hyperbolic} for a modern treatment.
This gives rise to a map, let us call it here the shadowing map, from the space of Reeb trajectories 
in $S$ to the space of $g_b$-geodesics on $M$, which takes a Reeb trajectory $\gamma$ to the unique geodesic $c$ that shadows it in the sense that after projecting 
$\gamma$ down to the manifold $M$ there exist lifts of both curves to the universal covering $\widetilde{M}$ which remain uniformly bounded away from one another as time $t$ goes to 
both $+\infty$ and $-\infty$.   This can be interpreted as saying that the (lifted) Reeb trajectory is asymptotic, in forwards and backwards time, to the same pair of points on the boundary at infinity
of $\widetilde{M}$, as the geodesic it is shadowing.   We show that this shadowing map is always surjective.  It is injective precisely when the Reeb flow is divergent, more or less by definition.   

In the third step we adapt Gromov's approach for geodesic flows \cite{gromov2000three}, to show that this shadowing map between sets of trajectories can be realised by a map $F:S\to S_{g_b}M$ between the underlying manifolds that maps trajectories to trajectories.   This map $F$ is continuous and surjective, and moreover it is injective, almost by definition, if and only if the Reeb flow on $S$ has the divergence property.

\subsection{A word on notation}
We need to refer to curves in the closed manifold $M$, its universal covering $\Mtilde$, and also the hypersurfaces $S\subset T^*M$ and $\Stilde\subset T^*\Mtilde$.  
To avoid a proliferation of hats and bars on these curves we have decided on the following section-specific conventions, which are explained in each section: 

In Chapter 3, the top-decoration $\sim$ will generally indicate a lift from $M$ to $\Mtilde$.  So $\gamma$ may denote a Reeb trajectory in $S\subset T^*M$, and $\eta$ the projection down to $M$, and then $\tilde{\eta}$ a lift to $\Mtilde$.  

In Chapters 4 and 5, $\sim$ will indicate a lift from $\Mtilde$ to $T^*\Mtilde$, as the spaces $M$ and $S\subset T^*M$ play more of a side role.   So $\tilde{\gamma}$ may denote a Reeb trajectory in $\Stilde\subset T^*\Mtilde$, and $\gamma$ the projection down to $\Mtilde$.  Similarly, $\tilde{c}$ may denote a geodesic trajectory in $T^*\Mtilde$ and $c$ its projection to a geodesic in $\Mtilde$.   

In contrast, in Chapter 2 the curves under discussion are Hamiltonian time-$1$ chords, often as generators of a Floer chain complex, and so we denote these by letters like $x$ and $y$.  

\subsection*{Acknowledgements}
The results in this article originated in the doctoral thesis of the second author. 
The project grew out of discussions with Gerhard Knieper concerning hyperbolic rigidity of geodesic flows, and we are particularly grateful to him for many valuable discussions and insights during the development of this work.   We also thank Alberto Abbondandolo for many helpful conversations.   Both authors gratefully acknowledge partial support by the DFG through the Collaborative Research Center SFB/TRR 191 -- 281071066 (Symplectic Structures in Geometry, Algebra and Dynamics).

\section{Basic notions}\label{S:basic}

\subsection{Hamiltonian dynamics on cotangent bundles}\label{S:HD}
Let $M$ be a closed smooth manifold of dimension $n\in\N$.   Points in its cotangent bundle $T^*M$ will often be denoted by $(q,p)$, with $q\in M$ 
and $p\in T^*_qM$, and $\pi:T^*M\to M$ will denote the bundle projection.  The manifold $T^*M$ has the following canonical structures; the Liouville $1$-form $\theta$, the symplectic structure $\omega=d\theta$, and the Liouville vector field $Y$, which in local canonical coordinates $(q,p)\in T^*M$ are given by 
\[
	\theta = pdq := \sum_{i=1}^{n}p_idq_i,\qquad \omega=dp\wedge dq:=\sum_{i=1}^{n}dp_i\wedge dq_i,\qquad Y=p\partial_p:= \sum_{i=1}^{n}p_i\partial_{p_i}.
\]
One easily verifies the Liouville condition on $Y$, that 
\begin{equation}\label{E:Liouville condition}
					\omega(Y,\cdot)=\theta.     
\end{equation}
Each fiber 
\[
			L_q=T^*_qM=\pi^{-1}(q)
\]
is a Lagrangian submanifold of $(T^*M,\omega)$.   
A $1$-periodic Hamiltonian, i.e.\ a smooth function $H:\mathbb{T}\times T^*M\to\R$, $\mathbb{T}=\R/\Z$, produces the $1$-periodic Hamiltonian vector field $X_H$ on $T^*M$, characterized by $-dH=\omega(X_H,\cdot)$.   
In local coordinates 
\[
			X_H(t,q,p)=\sum_{i=1}^{n}\partial_{p_i}H\partial_{q_i}-\sum_{i=1}^{n}\partial_{q_i}H\partial_{p_i}, 
\]
in particular, 
\[
				D\pi(q,p)\big[X_{H}(q,p)\big]=\sum_{i=1}^{n}\partial_{p_i}H\partial_{q_i}=:\partial_pH(t,q,p)\in T_qM.
\]
By a \textbf{Hamiltonian chord} from $L_0=T^*_{q_0}M$ to $L_1=T^*_{q_1}M$ we mean a smooth solution $\gamma:[0,1]\to T^*M$ to 
\[
		\left\{\begin{aligned}
			\dot{\gamma}(t)&=X_{H}(\gamma(t)),\quad\forall t\in[0,1]\\
			\gamma(i)&\in L_i\quad\mbox{for }i=0,1. 
			\end{aligned}
			\right.
\]
If $[\alpha]$ is a relative homotopy class in $\pi_1(M;q_0,q_1)$ then we may say that $\gamma$ is a Hamiltonian chord representing $[\alpha]$ if 
the projection $\pi\circ\gamma:[0,1]\to M$ is homotopic with fixed ends to $\alpha$.   The Hamiltonian chords from $L_0$ to $L_1$ are formally the critical points 
of the Hamiltonian action functional 
\begin{equation}\label{E:action functional}
				\A_{H}(\gamma)=\int_{[0,1]}\gamma^*\theta\ - \int_{0}^{1}H(t,\gamma(t))\,dt
\end{equation}
on the space of smooth paths in $T^*M$ from $L_0$ to $L_1$.    If we restrict to paths representing a given $[\alpha]$ then the critical points 
are precisely the Hamiltonian chords representing $[\alpha]$.  

\subsection{Reeb dynamics on fiberwise star-shaped hypersurfaces}
We will call a smoothly embedded hypersurface $S\subset T^*M$ fiberwise star-shaped, or simply star-shaped, if it is strictly transverse to the Liouville vector field $Y$ on $T^*M$.    
In particular, for each $q\in M$, the hypersurface $S$ meets the fiber $T^*_qM$ transversally and the submanifold 
\[
	S_q:= S\cap T^*_qM
\]
is an $n-1$-dimensional sphere transverse to $Y|_{T^*_qM}$.   
The star-shaped condition on $S$ ensures that the Liouville form $\theta$ restricts (pulls back via inclusion $i:S\hookrightarrow T^*M$) to a contact form $$\lambda:=i^*\theta$$ on $S$, and each 
sphere $S_q$ is a Legendrian submanifold of $(S,\lambda)$, meaning that it is everywhere tangent to the contact structure $\xi:=\ker\lambda$. 
The associated Reeb vector field $R$ on $(S,\lambda)$ is defined by the conditions 
\[
							d\lambda(R,\cdot)=0, \qquad \lambda(R)=1.
\]
A \textbf{Reeb chord} from $S_{q_0}$ to $S_{q_1}$ is defined similarly, but it is overdetermined if we require it to have time-$1$: We mean a smooth solution $\gamma:[0,T]\to S$, $T>0$, to 
\[
		\left\{\begin{aligned}
			\dot{\gamma}(t)&=R(\gamma(t)),\quad\forall t\in[0,T]\\
			\gamma(i)&\in S_{q_i}\quad\mbox{for }i=0,1. 
			\end{aligned}
			\right.
\]
Similarly we say $\gamma$ represents $[\alpha]\in\pi_1(M;q_0,q_1)$ if $\pi\circ\gamma:[0,T]\to M$ is homotopic with fixed ends to a representative $\alpha$.   
The $\lambda$-action of such a Reeb chord is defined by 
\begin{equation}\label{E:Reeb action}
				\A_{\lambda}(\gamma)=\int_{[0,T]}\gamma^*\lambda,  
\end{equation}
which trivially coincides with the value $T$.   

\subsection{$2$-homogenous Hamiltonians}\label{S:hom Ham}
As preparation for using Floer homology let us reinterpret the Reeb vector field on the star-shaped hypersurface $S\subset T^*M$ in the standard way as the restriction of a Hamiltonian vector field on the whole cotangent bundle, of the unique $2$-homogenous time-independent Hamiltonian $H:T^*M\to\R$ satisfying 
\[
					S=H^{-1}(1/2).
\]
The $2$-homogeneity condition is 
\begin{equation}\label{E:H is hom}
								\qquad	H(q,cp)=c^2H(q,p)\qquad \forall c\geq 0, 
\end{equation}
for each $(q,p)\in T^*M$, which in general means that $H$ is only smooth on the complement of the zero section $\mathcal{O}\subset T^*M$.   
The existence of $H$ follows from $S$ being star shaped; we refer to it as the $2$-homogenous Hamiltonian generated by $S$.  
As is well known, the Hamiltonian vector field for $H$ restricts to the Reeb vector field on $S$: 
\begin{equation}\label{E:Ham=Reeb}
		\qquad X_H(z)=R(z)\qquad \forall z\in S.  
\end{equation}
Indeed, since $S$ is a regular energy surface for $H$ we know that the restriction $X_H|_{S}$ is tangent to $S$ and moreover lies in the kernel of $d\lambda$.    
Differentiating \eqref{E:H is hom} with respect to $c$ gives  
\begin{equation}\label{E:dH=2H}
			\qquad dH(z)[Y(z)]=2H(z)\qquad \forall z\in T^*M\backslash\mathcal{O}
\end{equation}
and so by \eqref{E:Liouville condition}
\begin{equation}\label{E:theta[X_H]=2H}
							\theta[X_H]=2H\qquad 
\end{equation}
on all of $T^*M$, since $\theta(z)[X_H(z)]=\omega(z)(Y(z),X_H(z))=dH(z)[Y(z)]=2H(z)$ for all $z\in T^*M\backslash\mathcal{O}$, while both sides of \eqref{E:theta[X_H]=2H} trivially 
coincide on $\mathcal{O}$.   Hence for each $z\in S$ we have $\lambda(z)[X_H(z)]=\theta(z)[X_H(z)]=2H(z)=1$, confirming \eqref{E:Ham=Reeb}.  

If $H$ is $2$-homogenous, the Hamiltonian action of a Hamiltonian chord $\gamma:[0,1]\to T^*M$ is simply
\begin{equation}\label{E:action energy}
							\mathbb{A}_H(\gamma)=H(\gamma)
\end{equation}
where $H(\gamma)$ abbreviates $H(\gamma(t))$ for any $t\in[0,1]$.   Indeed, from the definition 
$\A_{H}(\gamma)=\int_{[0,1]}\gamma^*\theta\ - H(\gamma)$ 
and since $\gamma^*\theta(t)=\theta(\gamma(t))[\dot{\gamma}(t)]=\theta(\gamma(t))[X_{H}(\gamma(t))]=2H(\gamma(t))$ using \eqref{E:theta[X_H]=2H}, 
this becomes $\A_{H}(\gamma)=2H(\gamma)-H(\gamma)=H(\gamma)$ as claimed.   

The symmetry in \eqref{E:H is hom} implies another relation between Hamiltonian chords and Reeb chords beyond restriction to $S$.     
Indeed from \eqref{E:H is hom} the Hamiltonian vector field is $2$-homogenous with respect to $p$ in the fiber component and $1$-homogenous with respect to $p$ in the ``horizontal'' component.   More precisely, for each $c\geq 0$, we have $X_H(q,cp)=c\sum_{i=1}^{n}\partial_{p_i}H(q,p)\partial_{q_i}-c^2\sum_{i=1}^{n}\partial_{q_i}H(q,p)\partial_{p_i}$ in local coordinates.   Hence if $\gamma=(q,p):\R\to T^*M\backslash\mathcal{O}$ is an $X_H$-trajectory then for each $c>0$ so is 
\begin{equation}\label{E:rescaling in fibers}
					\overline{\gamma}(t):=\big(q(ct),cp(ct)\big)
\end{equation}
an $X_H$-trajectory.   Moreover, if $\gamma$ is a time-$1$ chord, that is has domain $[0,1]$, then $\overline{\gamma}$ 
is a time-$1/c$ chord, that is has domain $[0,1/c]$.    If $\gamma$ lies in the energy surface $\{H\equiv r\}$, $r>0$, then 
$\overline{\gamma}$ lies in $\{H\equiv c^2r\}$.    Thus choosing $c=1/\sqrt{2r}$ we can arrange that $\overline{\gamma}$ 
lies in $S$ and is thus a Reeb chord.    To summarize, the rescaling in \eqref{E:rescaling in fibers} defines a $1-1$ correspondence 
\begin{equation}\label{E:correspondence}
					\gamma\mapsto \overline{\gamma}
\end{equation}
from the set of Hamiltonian time-$1$ chords $\gamma:[0,1]\to T^*M$ from $T^*_{q_0}M$ to $T^*_{q_1}M$ representing a homotopy class $[\alpha]$, 
to Reeb chords from $S_{q_0}$ to $S_{q_1}$ also representing $[\alpha]$.    Under the correspondence in \eqref{E:correspondence} the actions are 
related by 
\begin{equation}\label{E:action H and Reeb}
					\frac{1}{2}\mathbb{A}_{\lambda}(\overline{\gamma})^2=\mathbb{A}_{H}(\gamma).  
\end{equation}
Indeed, by \eqref{E:action energy} $\gamma$ lies in the energy surface $\{H\equiv r\}$, where $r=\mathbb{A}_{H}(\gamma)$, 
and so $\overline{\gamma}$ has domain $[0,1/c]$ where $c=1/\sqrt{2r}$, thus $\mathbb{A}_{\lambda}(\overline{\gamma})=1/c=\sqrt{2r}$.

\subsection{Minimal geodesics as minimal Reeb trajectories}\label{S:geodesics}
Recall that a geodesic $c:\R\to M$ on a Riemannian manifold $(M,g)$ is called minimal if it has a lift (equivalently for every lift) $\tilde{c}:\R\to\widetilde{M}$ 
to the universal cover, so that for each pair of times $s<t$ the restriction of $\tilde{c}$ to $[s,t]$ is a geodesic segment of minimal length between its end points in $\widetilde{M}$ with respect to the lifted metric.   
Equivalently: For each $s<t$ the restriction $c:[s,t]\to M$ is length minimizing in its relative homotopy class in $M$ with fixed end points, i.e. 
\begin{equation}\label{E:length minimizing}
								\ell_{g}\big(c|_{[s,t]}\big)=\inf_{d}\ell_{g}\big(d\big)
\end{equation}
where the infimum is taken over all geodesic segments $d:I\to M$ from $c(s)$ to $c(t)$ that represent the relative homotopy class $[c|_{[s,t]}]$ in 
$\pi_1(M;q_0,q_1)$, and where $\ell_g(c)$ denotes the length of $c$ with respect to $g$; $\ell_g(c):=\int_{I}\|\dot{c}(r)\|_{g}\,dr$.   Of course \eqref{E:length minimizing} 
continues to hold if we restrict the infimum to unit-speed geodesics.   We will call a geodesic segment achieving the infimum in \eqref{E:length minimizing} 
a length minimizing geodesic segment.   

The geodesic segments from $q_0$ to $q_1$ in $M$ are also (up to reparameterization) the critical points of the Lagrangian 
energy functional 
\begin{equation}\label{E:energy}
					\mathcal{E}:\Omega(M;q_0,q_1)\to\R,\qquad \mathcal{E}(c):=\frac{1}{2}\int_{0}^{1}\|\dot{c}(t)\|^2_{g}\,dt
\end{equation}
on the Hilbert manifold of $W^{1,2}$ curves $[0,1]\to M$ from $q_0$ to $q_1$.  Since geodesics and critical points of $\mathcal{E}$ have constant 
positive speed, we see that at a critical point $c$ we have $\mathcal{E}(c)=\frac{1}{2}\ell_g(c)^2$, echoing \eqref{E:action H and Reeb}.   Note that by 
definition critical points of $\mathcal{E}$ are time-$1$ curves, whereas a geodesic may be a curve $c:[0,T]\to M$.    In this case, we can of 
course consider the reparameterized curve $\overline{c}:[0,1]\to M$ defined by $\overline{c}(t):=c(Tt)$ which is also a geodesic, and one finds that $\overline{c}$ is a critical point of $\mathcal{E}$ with $\mathcal{E}(\overline{c})=\frac{1}{2}\ell_g(\overline{c})^2=\frac{1}{2}\ell_g(c)^2$.  

The geodesic flow $\phi_g:\R\times TM\to TM$ of $g$ is characterized by the property that each trajectory $\sigma:\R\to TM$ 
has the form
\begin{equation}\label{E:gflow}
					\sigma(t)=\big(c(t),\dot{c}(t)\big)\in T_{c(t)}M
\end{equation}
for some $g$-geodesic $c:\R\to M$.    
The dual map (Legendre transform) $TM\to T^*M$, $v\mapsto g(v,\cdot)$ takes the geodesic flow to a flow called the co-geodesic flow on $T^*M$.  As is well known, the co-geodesic flow is also the Hamiltonian flow of the Hamiltonian 
\begin{equation}\label{E:G}
		G:T^*M\to\R, \qquad G(q,p):=\frac{1}{2}\|p\|^2_{g_*}, 
\end{equation}
where $g_*$ denotes the dual metric to $g$, that is, the induced fiberwise inner product on $T^*M$.   
This implies that each Hamiltonian trajectory of $G$ projects under $\pi:TM\to M$ to a geodesic $c:\R\to M$ for $g$.    Indeed, consider a $G$-Hamiltonian trajectory 
$\gamma:\R\to T^*M$.   Then $\gamma$ is Legendre-dual to a trajectory $\sigma:\R\to TM$ of the geodesic flow for $g$, 
which by \eqref{E:gflow} projects to a $g$-geodesic $c:\R\to M$.    Since the dual map $TM\to T^*M$ fixes the zero section pointwise - after the obvious 
identification of each zero section with $M$, we have $\pi\circ\gamma=\pi\circ\sigma=c$ as claimed.    Moreover the (constant) speed of $c$ is related to the 
energy of $\gamma$ as follows: 
\[
					\|\dot{c}(t)\|_{g}=\sqrt{2G(\gamma(t))}.
\]
Indeed, $\|\dot{c}(t)\|_{g}=\|\sigma(t)\|_{g}=\|\gamma(t)\|_{g_*}$, the second equality because the dual map $TM\to T^*M$ is an isometry on each fiber, 
and then we have $G(\gamma(t))=\|\gamma(t)\|^2_{g_*}/2$.   In particular, if $x:[0,1]\to T^*M$ is a $G$-Hamiltonian chord from $T^*_{q_0}M$ to $T^*_{q_1}M$ 
lying in the energy surface $\{z\in T^*M \,|\, G(z)=e\}$, then $\pi\circ x:[0,1]\to M$ is a $g$-geodesic from $q_0$ to $q_1$ with speed $\sqrt{2e}$, so that 
the length of $\pi\circ x$ is 
\begin{equation}\label{E:length action}
					\ell_{g}\big(\pi\circ x\big)=\sqrt{2e}=\sqrt{2\A_{G}(x)}
\end{equation}
using \eqref{E:action energy} to make the relation to action.   

We also see that $G$ is just the $2$-homogenous Hamiltonian generated by the unit (co)-sphere bundle 
\[
							S=\big\{(q,p)\in T^*M\ |\ \|p\|_{g_*}=1\big\}=G^{-1}(1/2)
\]
and therefore the co-geodesic flow restricted to $S$ is just the Reeb flow with respect to $\lambda=d\theta|_{S}$, see \ref{S:hom Ham}.  
Consider a Reeb trajectory $\gamma:\R\to S$.   Then as a $G$-Hamiltonian trajectory within the energy surface where $G\equiv 1/2$, 
we see by the discussion above that the projection $\pi\circ\gamma=c$ is a $g$-geodesic with speed $\sqrt{2G(\gamma(t))}=1$.   
Thus, for each $s<t$ we have $\mathbb{A}_{\lambda}(\gamma|_{[s,t]})=t-s=\ell_{g}(c|_{[s,t]})$.    So, the condition \eqref{E:length minimizing} for 
$c$ to be a minimal geodesic is precisely equivalent to the action minimizing condition  
\begin{equation}\label{E:action minimizing}
				\qquad\qquad	\mathbb{A}_{\lambda}\big(\gamma|_{[s,t]}\big)=\inf_{\eta}\mathbb{A}_{\lambda}(\eta)\qquad \forall s<t, 
\end{equation}
where the infimum is taken over all Reeb chords $\eta:I\to S$ from $S_{q_0}\subset T^*_{q_0}M$ to $S_{q_1}\subset T^*_{q_1}M$, 
where $q_0:=\pi(\gamma(s)),q_1:=\pi(\gamma(t))$, for which $\pi\circ\eta$ is homotopic to $\pi\circ\gamma|_{[s,t]}$ with ends fixed at $q_0$ and $q_1$.     
More generally, in the context of a general star-shaped hypersurface $S\subset T^*M$ not necessarily coming from a geodesic flow, we will call a 
Reeb chord achieving the infimum in \eqref{E:action minimizing} an \textbf{action minimizing Reeb chord}, and we define: 

\begin{definition}\label{D:minimal Reeb}
Let $S\subset T^*M$ be a fiberwise star-shaped hypersurface.    A Reeb trajectory $\gamma:\R\to S$ is said to be \textbf{minimal} 
if for each $s<t$ the chord $\gamma|_{[s,t]}$ is action minimizing between the Legendrian spheres $S_q, S_{q'}$ that it connects, 
in the sense of \eqref{E:action minimizing}.   We say that the Reeb flow (or the hypersurface $S$) is minimal if each trajectory is minimal in this sense.   
\end{definition} 

So by the above discussion, if $S$ is a unit co-sphere bundle, so that its Reeb flow comes from a geodesic flow on $TM$  via the Legendre transform, then a Reeb trajectory in $S$ is minimal in the sense of Definition \ref{D:minimal Reeb} precisely if it comes from a minimal geodesic on $M$.   

It would be natural therefore to try to generalise the rigidity theorem of Gromov \cite{gromov2000three} from geodesic flows to Reeb flows in which the minimal geodesic condition is replaced by 
the minimal action condition in Definition \ref{D:minimal Reeb}.    However, as mentioned in the introduction, this notion of minimality is very restrictive: 

\begin{proposition}\label{propn:minimal implies convex intro}
Let $S\subset T^*M$ be a fiberwise star-shaped hypersurface.  If the associated Reeb flow is minimal (in the sense of Definition \ref{D:minimal Reeb}) then 
$S$ is fiberwise convex.  
\end{proposition}

For a proof of Proposition \ref{propn:minimal implies convex intro} see 
Appendix \ref{S:Appendix}.   Thus effectively every minimal Reeb flow is just (the Legendre dual to) a Finsler geodesic flow.   
For this reason we looked for a weaker condition than Definition \ref{D:minimal Reeb}) 
that does not force the fibers to be convex but is strong enough to generalize the Morse-Gromov theorem.   This lead us to the notion of coarse minimality 
discussed in the next section.  

\subsection{Coarsely minimal Reeb flows}\label{S:coarse min}

In this article we introduce the following notion: 

\begin{definition}\label{D:CM}
Let $S\subset T^*M$ be a fiberwise star-shaped hypersurface and $T\geq 0$.    We call a Reeb trajectory $\gamma:\R\to S$ \textbf{coarsely minimal} (or $T$-coarsely minimal) 
if the following holds: For each $s<t$ there exists $s'<t'$ with $|s-s'|\leq T$ and $|t-t'|\leq T$ so that the chord $\gamma|_{[s',t']}$ is action minimizing in the sense of \eqref{E:action minimizing}.   
\end{definition} 

Thus $T$-coarsely minimal allows for chords $\gamma|_{[s,t]}$ which are not action minimizing (if $T>0$), but by adjusting the end points at a cost measured by $T$ we can always find a ``modified'' chord $\gamma|_{[s',t']}$ that is action minimizing in the sense described in Section \ref{S:geodesics}.  

\begin{definition}\label{D:CM flow}
A fiberwise star-shaped hypersurface $S\subset T^*M$, or its Reeb flow, is called coarsely minimal if there is some $T\geq 0$ so that each Reeb trajectory 
in $S$ is $T$-coarsely minimal in the sense above.   We may also say that $S$ or the flow is $T$-coarsely minimal if we wish to refer to $T$.  
\end{definition} 

Despite the title of this subsection, note that the notion of minimality in Definition \ref{D:CM flow} is not a condition purely on the contact manifold $(S,\lambda)$,  since it makes crucial use of the ambient cotangent bundle structure.   

From the discussion in \ref{S:geodesics}: If $S$ corresponds to a minimal geodesic flow then it is $0$-coarsely minimal.   Note also, that 
while for Reeb flows it is possible for a trajectory to be $T$-coarsely minimal without being $T'$-coarsely minimal for some $0\leq T'<T$, this cannot happen for geodesic flows.   Indeed, if $S$ comes from a geodesic flow (via the Legendre transform) and $\gamma:\R\to S$ is a 
$T$-coarsely minimal trajectory for some $T>0$, then it is automatically $0$-coarsely minimal, i.e.\ minimal in the sense of Definition \eqref{D:minimal Reeb}.  
This is simply because a minimal geodesic segment restricted to a subinterval is also minimal, but an action minimizing Reeb chord need not be action minimizing when restricted to a subinterval.  

As mentioned in the introduction, coarse-minimality is a weak enough condition that it does not force 
the fibers of $S\subset T^*M$ to be convex.   In Theorem \ref{thm:example intro} we construct a fiberwise star-shaped hypersurface $S$ for which 
the Reeb flow is $T$-coarsely minimal for some $T>0$ and some of the fibers are not convex.

\subsection{Quasi-geodesics and Morse's stability Lemma}\label{S:QG}
We recall Morse's Lemma \cite{morse1924fundamental}, as formulated for quasi-geodesics in the universal covering of a negatively curved Riemannian manifold as in Knieper \cite{knieper2002hyperbolic}. 

If $(X,d)$ is a metric space then a quasi-geodesic is any quasi-isometric embedding $\tilde{\eta}:(\R,|\cdot|)\to (X,d)$ of the real line.   Which means that there 
exist $A\geq 1,a\geq 0$ so that for all $s,t\in \R$ there holds: 
\begin{equation}\label{E:QG}
						\frac{1}{A}|s-t|-a\leq d\big(\tilde{\eta}(s),\tilde{\eta}(t)\big) \leq A|s-t|+a.  
\end{equation}
We will also refer to $\tilde{\eta}$ as an $(A,a)$-quasi-geodesic.   

If $g,g'$ are two Riemannian metrics on $M$ then $\tilde{\eta}:\R\to\widetilde{M}$ will be a quasi-geodesic with respect to the lift of $g'$ if and only if 
it is a quasi-geodesic with respect to the lift of $g$.   Thus we will sometimes speak of a curve in $M$ as lifting to a quasi-geodesic in $\widetilde{M}$ without referencing a metric.   

\begin{remark}\label{R:QG-coarsely}
Suppose a curve $\tilde{\eta}:\R\to (X,d)$ satisfies the quasi-geodesic conditions \eqref{E:QG} only in the following coarse sense:  
For all $s,t\in\R$ we can find $s',t'\in\R$ with $|s-s'|\leq T$ and $|t-t'|\leq T$, so that \eqref{E:QG} holds at $s'$ and $t'$.   Then at the expense of 
replacing $a$ by some $a'\geq a$, where $a'$ depends only on $d, A$ and $a$, one can show that $\tilde{\eta}:\R\to (X,d)$ 
is an $(A,a')$-quasi-geodesic.   For details see Lemma 4.4 and Remark 4.5 in \cite{dePooter_thesis}.  
\end{remark}

\begin{theorem}[Morse's Lemma]\label{T:Morse}
Let $(M,g)$ be a closed Riemannian manifold with negative sectional curvature bounded from above by $k_g<0$.   If $\tilde{\eta}:\R\to \widetilde{M}$ is a quasi-geodesic in the universal covering then there is a unique geodesic (up to reparameterization) $\tilde{c}:\R\to\widetilde{M}$ which remains uniformly 
bounded away from $\tilde{\eta}$, in the sense that the Hausdorff distance satisfies 
\begin{equation}\label{D:shadow}
							d_{H}\big(\tilde{c}(\R),\tilde{\eta}(\R)\big)<\infty. 
\end{equation}
Moreover, if $\tilde{\eta}$ is an $(A,a)$-quasi-geodesic, then $d_{H}\big(\tilde{c}(\R),\tilde{\eta}(\R)\big)\leq R$, for some $R$ depending only on $k_g,A,a$.  
We will say that $\tilde{c}$ shadows $\tilde{\eta}$ when \eqref{D:shadow} holds.   
\end{theorem}

For a proof see \cite{knieper2002hyperbolic} or \cite{guillarmou2026inverse}. 
The following Lemma will be used in the proof of Theorem \ref{T:main}.   Again, let $(\Mtilde,g)$ be the universal covering of a closed Riemannian manifold with sectional curvatures.   Recall that $\Mtilde$ is diffeomorphic to a ball via the exponential map, since there are no conjugate points, and we denote by 
$\partial\Mtilde$ the sphere at infinity and $\Mbar:=\Mtilde\cup\partial\Mtilde$ with the usual topology, see for example \cite{knieper2002hyperbolic} or \cite{bridson1999metric}, or \cite{guillarmou2026inverse}.  Recall that the topology is such that if $z_k\in\Mtilde$ is a sequence converging to a point at infinity $z\in\partial\Mtilde$ and $z'_k$ is another sequence such that $d_{g}(z_k,z'_k)$ is uniformly bounded, then $z'_k\to z$ also. 

\begin{lemma}\label{L:basic convergence 1}
Let $\tilde{c}:\R\to\Mtilde$ be a unit speed geodesic asymptotic to $x,y\in\partial\Mtilde$ on the boundary at infinity; say with $x=\lim_{t\to-\infty}\tilde{c}(t)$ and $y=\lim_{t\to\infty}\tilde{c}(t)$.   
Suppose that $x_k$ and $y_k$ are sequences in $\Mtilde$ with $x_k\to x$ and $y_k\to y$ as $k\to\infty$, 
and let $\tilde{c}_k:\R\to \widetilde{M}$ be a sequence of unit speed geodesics passing through $x_k$ and $y_k$ (in that order).   
Then there exists a sequence of shifts $s_k\in\R$ so that 
$\tilde{c}_k(\,\cdot +s_k)$ converges to $\tilde{c}$ in $C^{\infty}_{\mathrm{loc}}(\R,\Mtilde)$. 
\end{lemma}

The proof is a standard argument using Morse's Lemma along the following lines; Morse implies that between $x_k$ and $y_k$ the curve $\tilde{c}_k$ lies in a uniform neighborhood of $\tilde{c}$ and so after appropriately shifting the parameterizations we can assume that a subsequence of $\tilde{c}_k$ converges in $C^{\infty}_{\mathrm{loc}}(\R,\Mtilde)$ to a geodesic $\tilde{d}$ (smooth dependence of initial conditions in ODE's).  But $\tilde{d}$ must also lie in a bounded neighborhood of $\tilde{c}$.   Using convexity of the function $t\mapsto d_{g}(\tilde{c}(t),\tilde{d}(t))$, see for example Lemma 16.4 in \cite{guillarmou2026inverse}, this implies that 
$\tilde{c}$ and $\tilde{d}$ are geometrically indistinct, i.e. the limiting geodesic is $\tilde{c}$.  From this one can show that there was no need to take a subsequence after all.  

\section{Lagrangian Floer homology for pairs of fibers in $T^*M$}\label{S:Floer}

We recall the definition and invariance properties of Lagrangian Floer homology of two fibers $T^*_{q_0}M, T^*_{q_1}M$ in a symplectic cotangent bundle 
$(T^*M,\omega)$, that is associated to a Hamiltonian that is quadratic outside of a compact set.   
Such Hamiltonians satisfy the conditions in \cite{abbon2006cotangent}, equivalently \cite{abbon2013legendre}, and we use the sign conventions of the former.   

\subsection{The Floer homology groups}\label{S:Def FH}

Let $M$ be a closed manifold 
and $H:\mathbb{T}\times T^*M\to \R$ be a smooth time-dependent Hamiltonian\footnote{According to \cite{abbon2006cotangent} one could also work with $H:[0,1]\times T^*M\to \R$.} with \textbf{quadratic growth near infinity}, which for this article will mean that $H$ is autonomous and $2$-homogenous in the sense of \ref{S:hom Ham} outside of some compact subset of $T^*M$.   (Then $H$ satisfies the more general conditions (H1) and (H2) from \cite{abbon2006cotangent}.)      
For $q_0,q_1\in M$ set 
$$L_i:=T^*_{q_i}M$$  
and denote by $\P(q_0,q_1)=\P(H;q_0,q_1)$ the set of all Hamiltonian time-$1$ chords $x:[0,1]\to T^*M$ from $L_{q_0}$ to $L_{q_1}$ as in \ref{S:HD}.    
The final condition on $H$, or on $(q_0,q_1)$, is that we assume it is \textbf{non-degenerate} for the pair of fibers $L_{q_0},L_{q_1}$ in the sense that for the time-$1$ map of the Hamiltonian flow 
we have $\phi^{1}_{H}(L_{q_0})$ intersects $L_{q_1}$ transversally.    In this case we will also say that $L_{q_0},L_{q_1}$ \textbf{is a non-degenerate pair for $H$}.  
A pair $L_{q_0},L_{q_1}$ is non-degenerate if and only if $q_1$ is a regular value for $\pi\circ\phi^1_{H}:L_{q_0}\to M$.    Thus by Sard's theorem, for each $q_0\in M$, the 
set of points $q_1\in M$ for which the pair $L_{q_0},L_{q_1}$ is non-degenerate for $H$ forms a set of full Lebesgue measure.   

Define $CF_*(H;q_0,q_1)$ to be the graded free Abelian group with coefficients in $\Z_2$ generated by the chords $\P(q_0,q_1)$, and grading given by the Maslov index 
$\mu:\P(q_0,q_1)\to \Z$ as in \cite{abbon2006cotangent}*{Section 1.2}  equivalently \cite{robbin1993maslov}.  
Recall that the Maslov index at $x$ measures the total amount that the linearized flow along $x$ rotates $L_{q_0}$ relative to the fibers of $T^*M$, normalized by 
subtracting $n/2$.    The boundary operator of the Floer chain complex 
\[
				\partial :CF_*(H;q_0,q_1)\to CF_{*-1}(H;q_0,q_1)
\]
is defined by counting Floer strips, that is, solutions $u\in C^{\infty}(\R\times[0,1],T^*M)$ of the perturbed Cauchy-Riemann equation with Lagrangian 
boundary conditions 
\begin{equation}\label{E:FE}
				\left\{\begin{aligned}
				\partial_su-J_t(u)\big(\partial_tu-X_H(t,u)\big)=0\\
				u(\R\times\{i\})\subset L_{q_i},\qquad i=0,1, 
				\end{aligned}\right. 
\end{equation}
and finite energy 
\[
				E(u)	:=\int_{\R\times[0,1]}\left|\partial_su\right|^2\,dsdt  
\]
where $J=J_t$ is a smooth path of almost complex structures for $t\in[0,1]$, where each $J_t$ is compatible with $\omega=d\lambda$ in the sense 
that each $\omega(\cdot,J_t\cdot)$ defines a Riemannian metric.    The finite energy condition on $u$ and the non-degeneracy condition on $H$ ensure 
that there is a pair of Hamiltonian chords $x_{-}, x_{+}\in\P(q_0,q_1)$ so that $\lim_{s\to-\infty}u(s,t)=x_-(t)$ and $\lim_{s\to+\infty}u(s,t)=x_+(t)$.   
A standard computation shows that $\partial_s\A_H(u(s,\cdot))\leq 0$, so 
\begin{equation}\label{E:action decreases}
						\A_{H}(x_-)\geq \A_{H}(x_+), 
\end{equation}
with equality if and only if $u(s,\cdot)$ is independent of $s$, in particular if and only if $x_{-}=x_{+}$. 

Since $T^*M$ is not compact let us recall briefly the assumptions on the almost complex structure that we make, following \cite{abbon2006cotangent}, to ensure that this 
construction of Floer homology is well defined and suitably invariant.   
Fix a background Riemannian metric $g$ on $M$.   This determines a preferred $\omega$-compatible almost complex structure $\hat{J}$ on $T^*M$; 
with respect to the horizontal and vertical splitting $T_{(q,p)}T^*M=T^h_{(q,p)}T^*M\oplus T^v_{(q,p)}T^*M$, and using that both the horizontal and vertical 
spaces can be canonically identified with $T_{q}M$ and hence with each other, $\hat{J}$ has the form 
\begin{equation}\label{E:Jhat}
        \hat{J}=\begin{pmatrix}
        \ \ \, 0 & I \\
        -I & 0
        \end{pmatrix}.
\end{equation}
Then there exists \cite{abbon2006cotangent}*{Thm 1.15} a constant $j(H)>0$, depending on $H$, so that if the time-dependent almost complex structure $J$ 
satisfies 
\begin{equation}\label{E:close to Jhat}
			\|J-\hat{J}\|_{L^\infty}<j(H), 
\end{equation}
then the solutions to \eqref{E:FE} satisfy $L^{\infty}$-bounds, and so for each pair of chords $x_-,x_+$ 
the moduli space $\mathcal{M}(x_-,x_+)$ of Floer strips connecting $x_-$ to $x_+$, is compact modulo translation in the domain by a shift in the $\R$-coordinate.   
Moreover there is a Baire set\footnote{See \cite{abbon2006cotangent} for a description of the relevant almost complex structures as a complete metric space.}  of compatible almost complex structures satisfying \eqref{E:close to Jhat}, called regular almost complex structures, for which the moduli 
spaces defining the boundary operator are moreover smooth manifolds.   Combining this with the compactness yields that the boundary operator applied to a chord 
$x$ involves only a finite sum and is therefore well defined, and via the standard gluing arguments one can show $\partial^2=0$.   
The homology of this complex is called the Lagrangian Floer homology of the pair of fibers 
over $q_0$ and $q_1$ and is denoted by $$HF_*(H,J;q_0, q_1),$$ or sometimes just by $HF_*(q_0, q_1)$.  
There is also a natural filtration through relative homotopy classes of paths from $q_0$ to $q_1$.   That is, if 
\[
			[\alpha]\in \pi_1(M;q_0,q_1)
\] 
let $CH_k(H,q_0,q_1,[\alpha])$ be the subgroup of $CH_k(H,q_0,q_1)$ generated by chords whose projection under $\pi:T^*M\to M$ represent $[\alpha]$.       
Clearly solutions to \eqref{E:FE} connect chords that project to the same homotopy class in $\pi_1(M;q_0,q_1)$, so we obtain a subcomplex, and we denote 
the associated homology groups by $HF_*(H,q_0,q_1,[\alpha])$.

 The construction of Floer homology for closed symplectic manifolds is originally due to A. Floer (see \cites{floer1988aMorse, floer88brelative, floer88cunregularized, floer1989symplectic, floer1989bcomplex}), 
 while for cotangent bundles there are constructions due to Viterbo \cite{viterbo1999functors, viterbo2003functors}, Biran-Polterovich-Salamon \cite{biran2003propagation} Salamon-Weber \cite{Salamon_2006}, and Abbondandolo-Schwarz \cite{abbon2006cotangent} which is the reference we are most closely following.   
%all of which can be viewed as a particular case of symplectic homology for Liouville domains (see [FH94, CFH95, Vit99, Oan04, Sei08a, BO09]).

\subsection{Invariance}\label{S:invariance}
In the previous section we described conditions under which the Lagrangian Floer homology groups are well defined.   To summarize:   
If $H$ satisfies the non-degeneracy condition for the fibers $L_{q_0}$ and $L_{q_1}$ and the quadratic growth condition, then there exists a $j(H)>0$ so that the 
groups $HF_*(H,J,q_0,q_1)$ 
are well defined for a Baire set of so called regular almost complex structures satisfying \eqref{E:close to Jhat}.   
In this section we recall to what extent these groups are independent of $J$ and $H$ satisfying the above conditions.   

If $J_0$ and $J_1$ are two regular almost complex structures satisfying \eqref{E:close to Jhat} then there is an isomorphism $\phi_{01}$ of the chain complexes 
associated to $(H,J_0,q_0,q_1)$ and $(H,J_1,q_0,q_1)$ 
which is unique up to chain homotopy \cite{cornea2003rigidity}, 
see also \cite[Thm.~1.19]{abbon2006cotangent}. 

In particular the associated homology groups are canonically isomorphic and can be denoted by 
$HF_*(H,q_0,q_1)$, suppressing the almost complex structure.  

If $H_0$ and $H_1$ are two Hamiltonians on $\mathbb{T}\times T^*M$, non-degenerate with respect to $q_0$ and $q_1$, and both satisfying the quadratic growth at infinity 
condition, then for any almost complex structure $J$ that is regular for both $H_0$ and $H_1$ and which satisfies \eqref{E:close to Jhat} for both $H_0$ and $H_1$, 
there is a homotopy equivalence $\psi_{01}$ of the chain complexes associated to $(H_0,J,q_0,q_1)$ and $(H_1,J,q_0,q_1)$, 
\[
				\psi_{01}:CF_*(H_0;q_0,q_1)\rightarrow CF_*(H_1;q_0,q_1),   
\]
uniquely determined up to chain homotopy.  In particular, there is a canonical isomorphism 
\[
						\Psi_{01}:HF_*(H_0,q_0,q_1)\rightarrow HF_*(H_1,q_0,q_1)  
\]
called the continuation map.    The continuation isomorphisms are natural in the sense that if $H_2$ is a third Hamiltonian (that is non-degenerate  
with respect to $q_0$ and $q_1$ and quadratic outside a compact set), then $\Psi_{12}\circ\Psi_{01}=\Psi_{02}$.   Moreover on the chain level $\psi_{12}\circ\psi_{01}$ 
is chain homotopic to $\psi_{02}$ (after choosing a $J$ that for all three Hamiltonians satisfies \eqref{E:close to Jhat} and is regular).   

Recall that $\psi_{01}$ is uniquely determined by a choice of homotopy $(H_s)_{s\in\R}$ of smooth Hamiltonians from $H_0$ to $H_1$.    
That is, if $H_s=H_0$ for $s\leq 0$ and $H_s=H_1$ for all $s\geq 1$, then the associated morphism $\psi_{01}$ takes each generator $x$ to a sum over 
generators $y$, for which there exists a solution $u:\R\times[0,1]\to T^*M$ to the $s$-dependent Floer equation 
\begin{equation}\label{E:FE with s}
				\partial_su(s,t)-J_{t}(u(s,t))\big(\partial_tu(s,t)-X_{H_{s,t}}(t,u(s,t))\big)=0
\end{equation}
with Lagrangian boundary conditions $u(\R\times{i})\subset T^*_{q_i}M$, $i=0,1$, for which $\lim_{s\to-\infty}u(s,t)=x(t)$ and $\lim_{s\to+\infty}u(s,t)=y(t)$.    
If moreover $$H_0\leq H_1$$ then we may choose the homotopy $(H_s)_{s\in\R}$ to be monotone, that is with $\partial_sH\geq0$ for all $s\in\R$.
Then a standard computation using \eqref{E:FE with s} shows that the path of actions $s\mapsto\A_{H_s}(u(s,\cdot))$ is monotonic decreasing, and in particular 
\begin{equation}\label{E:action decreasing}
				\A_{H_0}(x)\geq  \A_{H_1}(y), 
\end{equation}
with strict inequality unless $y=x$. 

Finally, everything in the above discussion passes over word for word when we filter through a relative homotopy class $[\alpha]\in\pi_1(M;q_0,q_1)$.   

\section{Coarsely minimal Reeb trajectories project to quasi-geodesics} 
Let $(M,g)$ be a closed Riemannian manifold and $(S,\lambda)$ a fiberwise star-shaped hypersurface of $T^*M$ equipped with the contact form $\lambda=d\theta|_{S}$ 
obtained by restricting the Liouville form.   For convenience let us assume throughout this section that $g$ has negative 
sectional curvatures.   This is not really necessary, see Remark \ref{R:negative curvature}, but it makes the proofs significantly simpler 
and we will require negative curvature in the application anyway.    

In this section we show: 

\begin{theorem}\label{T:minimal implies QG}
If $\gamma:\R\to S$ is a $T$-coarsely minimal Reeb trajectory, 
then the projection $\eta=\pi\circ\gamma:\R\to M$ lifts to a quasi-geodesic $\tilde{\eta}:\R\to\widetilde{M}$ in the universal covering with respect to the lifted metric $\tilde{g}$.  
\end{theorem}

Recall from \ref{S:QG} that this means there exist $A\geq 1,a\geq 0$ so that for all $s,t\in \R$ there holds: 
\begin{equation}\label{E:QG2}
						\frac{1}{A}|s-t|-a\leq d_{\tilde{g}}\big(\tilde{\eta}(s),\tilde{\eta}(t)\big) \leq A|s-t|+a.  
\end{equation}
We will see that $A,a$ can be chosen to depend only on $S, g, T$.   Indeed, we will show that the upper bound in \eqref{E:QG2} holds if $a=0$ and $A=\max\|d\pi(z)[R(z)]\|_{g}$ 
where $R$ denotes the Reeb vector field on $S$, while the lower bound in \eqref{E:QG2} holds if $A$ is sufficiently large that $S\subset T^*M$ is contained in the $g_*$-disc bundle of 
radius $A$, and $a:=2(T+M)/A$ where $M=\max\{d_{\tilde{g}}(\tilde{\eta}(s),\tilde{\eta}(s'))\}$ over all $s,s'$ with $|s-s'|\leq T$.   In particular $A$ is also independent of $T$.

The idea in the proof of Theorem \ref{T:minimal implies QG} is that the non-constant terms in \eqref{E:QG2} can be represented as actions of chords, and the inequalities 
can be deduced from the existence of connecting Floer strips; $|s-t|$ is the $\lambda$-action 
of $\gamma$, while 
$d_{\tilde{g}}\big(\tilde{\eta}(s),\tilde{\eta}(t)\big)$ is the Hamiltonian action of the minimal $G$-chord from $L$ to $L'$ that is homotopic to $\gamma|_{[s,t]}$, where 
$G$ is the kinetic Hamiltonian from $g$ as in \eqref{E:G}.  

\subsection{Proof of the upper bound in \eqref{E:QG2}} 
This is the trivial part of \eqref{E:QG2}, following just from compactness of $S$.   

\begin{lemma} 
Let $(S,\lambda)$ be a fiberwise star-shaped hypersurface of $T^*M$ equipped with the contact form from restriction of the Liouville form.    Then for 
each Reeb trajectory $\gamma:\R\to S$ and all $s,t\in\R$ the right hand inequality in \eqref{E:QG2} holds with $a=0$ and some $A>0$ depending 
only on $S$ and $g$.   
\end{lemma}
\begin{proof} 
Let $R$ denote the Reeb vector field on $(S,\lambda)$.   Then by compactness of $S$, the function 
\begin{equation}\label{E:projected Reeb}
	S\ni z\mapsto \|d\pi(z)[R(z)]\|_{g}
\end{equation} 
is uniformly bounded.    Denote the maximum by $A>0$.    
Suppose $\gamma:\R\to S$ is a Reeb trajectory.   Let $\eta:=\pi\circ\gamma:\R\to M$ denote the projection down to $M$, and let $\tilde{\eta}:\R\to\widetilde{M}$ be any lift of $\eta$ to the universal cover of $M$.   Then for each $s<t$ we have 
\begin{equation}
	d_{\tilde{g}}(\tilde{\eta}(s),\tilde{\eta}(t))\leq\ell_{\tilde{g}}\big(\tilde{\eta}|_{[s,t]}\big)=\ell_{g}(\eta|_{[s,t]})=\int_{s}^{t}\|\dot{\eta}(t)\|_{g}\,dt.  
\end{equation}
Now from $\|\dot{\eta}(\tau)\|_{g}=\|d\pi(\gamma(\tau))[R(\gamma(\tau)]\|_{g}\leq A$ we conclude that $d_{\tilde{g}}(\tilde{\eta}(s),\tilde{\eta}(t))\leq A|t-s|$.  
Thus the right hand inequality in \eqref{E:QG2} holds with $a=0$ and $A$ equal to the maximum of the function \eqref{E:projected Reeb}.  
\end{proof}

\subsection{Proof of the lower bound in \eqref{E:QG2}} 
This part of \eqref{E:QG2} uses the Floer tools recalled in \ref{S:Floer}.     

Recall that $(M,g)$ is a closed Riemannian manifold with negative sectional curvatures and $(S,\lambda)$ is a fiberwise star-shaped hypersurface of $T^*M$ equipped with the contact form from 
restriction of the Liouville form.     Let $H:T^*M\to\R$ be the $2$-homogenous Hamiltonian for which 
\[
								S=\{z\in T^*M\ |\ H(z)=1/2\}.   
\]
Choose $C>0$ so that 
\begin{equation}\label{E:G<H}
							G(q,p):=\frac{1}{2C^2}\|p\|^2_{g_*}\leq H(q,p)				
\end{equation}
for all $(q,p)\in T^*M$, where $g_*$ denotes the induced fiberwise inner product on $T^*M$.    Thus $G:T^*M\to\R$ satisfies $G\leq H$.   Note that \eqref{E:G<H} just means 
that $S\subset T^*M$ is contained in the $g_*$-disc bundle of radius $C$.  

\begin{proposition}\label{L:existence of chord with bounds}  
Let $q_0,q_1\in M$ and $[\alpha]\in\pi_1(M,q_0,q_1)$ be a relative homotopy class.    Then there exists a Hamiltonian chord $x:[0,1]\to T^*M$ for $H$ from $T^*_{q_0}M$ 
to $T^*_{q_1}M$, whose projection $\eta:=\pi\circ x:[0,1]\to M$ represents $[\alpha]$, and so that 
\[
						\A_{H}(x)\leq \frac{1}{2}C^2d_{\tilde{g}}(\tilde{q}_0,\tilde{q}_1)^2,   
\]
where $C$ satisfies \eqref{E:G<H} and $\tilde{q}_0, \tilde{q}_1$ are points in the universal covering $\widetilde{M}$ so that some lift $\tilde{\eta}:[0,1]\to\widetilde{M}$ connects 
$\tilde{q}_0$ to $\tilde{q}_1$.  
\end{proposition}
\begin{proof}   
It is convenient to express $G$ as a kinetic Hamiltonian, that is, in the form 
\[
				G(q,p)=\frac{1}{2}\|p\|^2_{g'_*}
\]
where $g':=Cg$ is the conformally equivalent metric.  Note that $g'$ also has negative curvature, and that the dual is given by $g'_*=(1/C)g_*$.      
Hence each $G$-Hamiltonian trajectory is the Legendre-transform of a trajectory of the $g'$-geodesic flow on $TM$, as discussed in \ref{S:geodesics}.   
In particular, consider the $G$-chords from $T^*_{q_0}M$ to $T^*_{q_1}M$ whose projection to $M$ represent $[\alpha]$.   
There is precisely one such chord $y:[0,1]\to T^*M$, namely $y$ is the Legendre-transform of $y':[0,1]\to TM$ where $\pi\circ y':[0,1]\to M$ is the unique $g'$-geodesic from 
$q_0$ to $q_1$ representing $[\alpha]$ and parameterized to have domain the unit interval - uniqueness coming from the negative curvature assumption on $g'$.    

By \eqref{E:length action} we have $\A_{G}(y)= \frac{1}{2}\ell_{g'}(\pi\circ y)^2$, and since $\ell_{g'}(\pi\circ y)=\ell_{\tilde{g}'}(\widetilde{\pi\circ y})=d_{\tilde{g}'}(\tilde{q}_0,\tilde{q}_1)$ 
because $\pi\circ y=\pi\circ y'$ is length minimizing in its homotopy class, 
we conclude that 
\begin{equation}\label{E:action-distance-2}		
		\A_{G}(y)=\frac{1}{2}d_{\tilde{g}'}(\tilde{q}_0,\tilde{q}_1)^2=\frac{1}{2}C^2d_{\tilde{g}}(\tilde{q}_0,\tilde{q}_1)^2.   
\end{equation}
To apply Floer theory suppose first that the fibers $T^*_{q_0}M, T^*_{q_1}M$ form a non-degenerate pair for both $H$ and $G$ in the sense of \ref{S:Def FH}. 
We also need to smoothen $H$, so choose a small neighborhood $\U\subset T^*M$ of the zero section, so that 
all time-$1$ Hamiltonian chords for $H$ and $G$ from $T^*_{q_0}M$ to $T^*_{q_1}M$ are disjoint from $\U$ - this is possible since either $q_0$ and $q_1$ are distinct or 
they are equal and $[\alpha]\neq 0$ due to the non-degeneracy assumption on $H$ and $G$.   Now set 
$H_{\epsilon}=\beta\cdot H$, some smooth $\beta:T^*M\to[0,1]$ supported in $\U$ and vanishing on a smaller neighborhood of the zero section.   Then $H_{\epsilon}$ is a 
smoothening of $H$ that agrees with $H$ outside of $\U$, and 
\begin{equation}\label{E:H leq G}
										G_{\epsilon}\leq H_{\epsilon} 
\end{equation}
where $G_{\epsilon}=\beta\cdot G$.  
Consider the $G_{\epsilon}$-chords from $T^*_{q_0}M$ to $T^*_{q_1}M$ whose projection to $M$ represent $[\alpha]$.    By our choice of $\U$ these are precisely the $G$-chords from $T^*_{q_0}M$ to $T^*_{q_1}M$ representing $[\alpha]$.   By the discussion above there is only one such chord $y:[0,1]\to T^*M$ and $\A_{G_{\epsilon}}(y)=\A_{G}(y)$.  
Since $g'$ has negative curvature and therefore no conjugate points, and since the projection of $y$ to $M$ is a $g'$-geodesic segment, we know that the Maslov index of $y$ is zero.    
We conclude that the Lagrangian Floer chain complex $CF_*(G_{\epsilon},q_0,q_1;[\alpha])$ filtered through the homotopy class $[\alpha]$, has precisely one generator in 
degree $0$, namely $y$, while the homology groups in all other degrees are zero.     It follows that the corresponding Floer homology groups $HF_k(G_{\epsilon},q_0,q_1;[\alpha])$ 
vanish for $k\neq 0$ and has rank-$1$ in degree $k=0$.    

By the invariance property of Lagrangian Floer homology $HF_k(H_{\epsilon},q_0,q_1;[\alpha])=0$ for $k\neq 0$ 
and has rank-$1$ in degree $k=0$.   Moreover, the continuation map $\Psi:HF_0(H_{\epsilon},q_0,q_1;[\alpha])\to HF_0(G_{\epsilon},q_0,q_1;[\alpha])$ is an isomorphism, 
so $\Psi^{-1}([y])$ is non-zero, and so is represented by a cycle that includes a generator $x$ for $CF_0(H_{\epsilon},q_0,q_1;[\alpha])$.    
From \eqref{E:H leq G} and \eqref{E:action decreasing} we have $\A_{H_{\epsilon}}(x)\leq \A_{G_{\epsilon}}(y)$.  
However by our choice of $\U$, $x$ and $y$ must be disjoint from $\U$ and so $x$ is actually a Hamiltonian chord for $H$ disjoint from the zero section, so 
\[
						\A_{H}(x)=\A_{H_{\epsilon}}(x)\leq \A_{G_{\epsilon}}(y)=\A_{G}(y).   
\]
Combining this with \eqref{E:action-distance-2} proves the Proposition when the fibers $T^*_{q_0}M, T^*_{q_1}M$ form a non-degenerate pair for both $H$ and $G$.  

Finally, suppose that the fibers $T^*_{q_0}M, T^*_{q_1}M$ form a degenerate pair for either $H$ or $G$.   Then we find a sequence $q_1(k)\to q_1$ in $M$  
so that the pairs of fibers $T^*_{q_0}M, T^*_{q_1(k)}M$ are non-degenerate for both $H$ and $G$, see \ref{S:Def FH}.    Applying the argument above for a suitable sequence of homotopy classes 
$[\alpha_k]\in \pi_1(M,q_0,q_1(k))$ that converge to $[\alpha]$ in a suitable way, we obtain a sequence of $H$-Hamiltonian chords $x_k:[0,1]\to T^*M$ from $T^*_{q_0}M$ to $T^*_{q_1(k)}M$ 
so that $\A_{H}(x_k)\leq (1/2)C^2d_{\tilde{g}}(\tilde{q}_0,\tilde{q}_1(k))^2$ and $\pi\circ x_k$ represents $[\alpha_k]$.   It is easy to extract a convergent subsequence, since the fiber 
$S_{q_0}=T^*_{q_0}M\cap S$ is compact, and obtain a limiting chord satisfying the limiting inequality as $k\to\infty$.   This completes the proof of the Proposition.  
\end{proof}

Now we can prove the lower bound in \eqref{E:QG2}: 

\begin{lemma} 
Suppose $\gamma:\R\to S$ is a $T$-coarsely minimal Reeb trajectory.    Then there exist $A\geq 1, a\geq 0$ so that for all $s,t\in \R$ there holds: 
\begin{equation}\label{E:QG3}
						\frac{1}{A}|s-t|-a\leq d_{\tilde{g}}\big(\tilde{\eta}(s),\tilde{\eta}(t)\big)
\end{equation}
where $\tilde{\eta}:\R\to\widetilde{M}$ is any lift to the universal covering of $\eta=\pi\circ\gamma:\R\to M$.  
\end{lemma}
\begin{proof} 
Set $A=C>0$ from \eqref{E:G<H} and $M=\max\{d_{\tilde{g}}(\tilde{\eta}(s),\tilde{\eta}(s'))\}$ over all $s,s'$ with $|s-s'|\leq T$.   Then set $a:=2(T+M)/A$.   
Suppose $s<t$.   Since $\gamma$ is $T$-coarsely minimal we find $s'<t'$ with $|s-s'|\leq T$ and $|t-t'|\leq T$ so that the restriction 
$\gamma|_{[s',t']}$ has minimal action in the sense of Definition \eqref{E:action minimizing}.   From the triangle inequality: 
$|s-t|\leq 2T+|s'-t'|$ and $d_{\tilde{g}}(\tilde{\eta}(s'),\tilde{\eta}(t'))\leq d_{\tilde{g}}(\tilde{\eta}(s),\tilde{\eta}(t))+2M$.   Thus 
\[
					\frac{1}{A}|s-t|-a\leq \frac{1}{A}|s'-t'|-2M.  
\]
So \eqref{E:QG3} will follow if we can show 
\begin{equation}\label{E:QG4}
						\frac{1}{A}|s'-t'|\leq d_{\tilde{g}}\big(\tilde{\eta}(s'),\tilde{\eta}(t')\big).  
\end{equation}
Thus our goal is to show \eqref{E:QG4} for all $s'<t'$ for which the restriction $\gamma|_{[s',t']}$ has minimal action in the sense of Definition \eqref{E:action minimizing}.   

Fix such a pair of parameters $s'<t'$.   Set $q_0:=\pi\circ\gamma(s')$ and $q_1:=\pi\circ\gamma(t')$.   
Set $G:T^*M\to\R$, $G(q,p)=\frac{1}{2C^2}\|p\|^2_{g_*}$ as in \eqref{E:G<H}.      
Let $[\alpha]\in\pi_1(M,q_0,q_1)$ be the relative homotopy class representing the projection of $\gamma|_{[s',t']}$ down to $M$.    Then by Proposition \ref{L:existence of chord with bounds} 
there exists a Hamiltonian $H$-chord $x:[0,1]\to T^*M$ from $T^*_{q_0}M$ to $T^*_{q_1}M$ satisfying 
\[
						\A_{H}(x)\leq \frac{1}{2}C^2d_{\tilde{g}}(\tilde{q}_0,\tilde{q}_1)^2.  
\]
The time-$1$ chord $x$ can be rescaled along the fibers, as in \eqref{E:correspondence}, to produce a chord $\bar{x}$ lying 
on the energy surface $S=\{z\in T^*M\ |\ H=1/2\}$, and which is therefore a Reeb chord for $(S,\lambda)$, with action changing according to \eqref{E:action H and Reeb}, 
namely $\frac{1}{2}\mathbb{A}_{\lambda}(\overline{x})^2=\mathbb{A}_{H}(x)$.   Thus 
\[
						\mathbb{A}_{\lambda}(\overline{x})\leq Cd_{\tilde{g}}(\tilde{q}_0,\tilde{q}_1).  
\]
Since $\gamma|_{[s',t']}$ has minimal action in the sense of Definition \eqref{E:action minimizing}, we conclude that 
\[
			|s'-t'|=\mathbb{A}_{\lambda}(\gamma|_{[s',t']})\leq\mathbb{A}_{\lambda}(\overline{x})\leq Cd_{\tilde{g}}(\tilde{q}_0,\tilde{q}_1)=Cd_{\tilde{g}}\big(\tilde{\eta}(s'),\tilde{\eta}(t')\big), 
\]
proving \eqref{E:QG4} for $A=C$, and hence the Lemma.  
\end{proof} 

This completes the proof of the left hand side of \eqref{E:QG2} and with it the proof of Theorem \ref{T:minimal implies QG}

\begin{remark}\label{R:negative curvature}
Throughout this section it was not actually necessary to require $g$ has negative curvature.   It was merely convenient to see that the Floer homology groups are not all vanishing, 
which follows because each geodesic in $\tilde{q}_0$ and $\tilde{q}_1$ is unique, so that the chain complex for $G$ is zero in all but one degree.   More generally one could appeal 
to the isomorphism between the Floer homology of $G$ and the singular homology of the path space from $q_0$ to $q_1$ \cite{abbon2006cotangent,abbon2013legendre}, using that the connected component containing 
a curve representing $[\alpha]$ is non-empty.   
\end{remark}

\begin{remark}\label{R:existence of chords}
A special case of Proposition \ref{L:existence of chord with bounds} is that any two fibers $T^*_{q_0}M$ and $T^*_{q_1}M$ can be connected in either direction 
by a Reeb chord in $S$, since a Hamiltonian $H$-chord connecting these fibers can be rescaled as in \eqref{E:correspondence} to produce a Reeb chord.  
Since this can be done while prescribing the homotopy class of the projection to $M$, it follows 
that if $\Stilde\subset T^*\Mtilde$ denotes the covering of $S$ that lies in the cotangent bundle of the universal covering of $M$, then any 
two fibers $T^*_{q_0}\Mtilde$ and $T^*_{q_1}\Mtilde$ can be connected in either direction by a Reeb chord in $\Stilde$.    This generalization of Hopf-Rinow to Reeb flows using Floer homology is well known going back at least to  \cite{maca2011spherization}.  
\end{remark}

\section{The shadowing map}\label{S:shadow}
Throughout this section $(M,g)$ is a closed Riemannian manifold with negative sectional curvatures.   The curvature assumption will be used to apply Morse's Lemma.   
Let $S\subset T^*M$ be a fiberwise star-shaped hypersurface and denote by $\lambda=d\theta|_{S}$ the contact form obtained by restricting the Liouville form.   
We write $\widetilde{M}$ for the universal covering of $M$ and $\widetilde{S}\subset T^*\Mtilde$ for the corresponding covering of $S$.    Note that we can also speak of the 
Reeb flow on $\Stilde$. 
Distances in $\Mtilde$ will always mean with respect to the lifted metric $\tilde{g}$.   

Suppose that the Reeb flow on $S$ is $T$-coarsely minimal for some $T\geq 0$, see Definition \ref{D:CM flow}.    Then by Theorem \ref{T:minimal implies QG} there exist $A,a>0$ 
depending only on $g, S, T$, so that each Reeb trajectory 
$\tilde{\gamma}:\R\to \Stilde$ projects down to an $(A,a)$-quasi-geodesic $\gamma=\pi\circ\tilde{\gamma}:\R\to\Mtilde$.   So by Morse's Lemma 
there is a unique geodesic (up to reparameterization) 
$c:\R\to\widetilde{M}$ which remains uniformly bounded away from $\gamma$ in the sense that 
the Hausdorff distance satisfies  
\begin{equation}\label{D:shadow 2}
							d_{H}\big(c(\R),\gamma(\R)\big)<\infty. 
\end{equation}
When \eqref{D:shadow 2} holds we will say that $c$ \textbf{shadows} $\gamma$ (or vice versa).   Note also that in this case Morse's Lemma gives an $R>0$, depending only 
on $g, S, T$, so that 
\[
							d_{H}\big(c(\R),\gamma(\R)\big)<R.  
\]
Let us write $\tilde{\gamma}\sim\tilde{\gamma}'$ if $\tilde{\gamma}$, $\tilde{\gamma}'$ are two Reeb trajectories in $\Stilde$ that are geometrically indistinct, that is for which $\tilde{\gamma}(\R)=\tilde{\gamma}'(\R)$.  
Similarly we write $\tilde{c}\sim\tilde{c}'$ if two geodesic trajectories in $S_{\tilde{g}}\Mtilde$ have the same image.  We denote equivalence classes by $[\tilde{\gamma}]$ etc.     
The Morse Lemma therefore defines a map between equivalence classes of Reeb trajectories 
to equivalence classes of geodesic trajectories: 
\[
				\tilde{f}:\big\{\tilde{\gamma}:\R\to \Stilde\,\big|\,\mbox{Reeb trajectory}\big\}/\!\sim\ \ \ \longrightarrow\ \  \big\{\tilde{c}:\R\to S_{\tilde{g}}\Mtilde\,\big|\,\mbox{geodesic trajectory}\big\}/\!\sim 
\]
defined to take $[\tilde{\gamma}]$ to the unique $[\tilde{c}]$ for which \eqref{D:shadow 2} holds.  The uniqueness of $[\tilde{c}]$ follows from the negative curvature assumption on $g$, as this implies the divergence property for the geodesic flow.   
We will refer to $\tilde{f}$ as the \textbf{shadowing map}.   To investigate the shadowing map $\tilde{f}$ we will use the orthogonal projection map onto a geodesic: 

\begin{definition}
Let $c:\R\to\Mtilde$ be a geodesic with respect to $\tilde{g}$.   Then we define $P_{c}:\Mtilde\to c(\R)$ to be the orthogonal projection 
onto the image of $c$.   That is, $P_{c}$ takes a point $q$ to the unique nearest point in $c(\R)$.  
\end{definition}

Note that $P_{c}$ is well defined and continuous: It is well defined because $\tilde{g}$ has negative curvature implies that $c$ is a minimal geodesic 
which implies that its image is a geodesically convex set, see \cite{knieper2002hyperbolic} or \cite{guillarmou2026inverse} for more details on $P_{c}$.   

\begin{proposition}\label{P:f onto}
The shadowing map $\tilde{f}$ is surjective.   It is injective precisely if the Reeb flow is divergent.  
\end{proposition}
\begin{proof} 
That injectivity of $\tilde{f}$ is equivalent to the Reeb flow being divergent is immediate from the definitions.   

It remains to show surjectivity.   Let $\tilde{c}:\R\to S_{\tilde{g}}\Mtilde$ be a geodesic trajectory.   We need to construct a Reeb trajectory $\tilde{\gamma}:\R\to\Stilde$ so that $\gamma$ and $c$, 
the respective projections down to $\Mtilde$, satisfy \eqref{D:shadow 2}.   
For each $n\in\N$ there exists a Reeb chord $\tilde{\gamma}_n:I_n\to\Stilde$ whose projection $\gamma_n$ down to $\Mtilde$ starts at $c(-n)$ and ends at $c(n)$, see Remark \ref{R:existence of chords}.   Hence $c^{-1}\circ P_{c}\circ\gamma_n:I_n\to\R$ defines a continuous function from the 
compact interval $I_n\subset\R$, and the image contains $-n$ and $n$.   By the intermediate value theorem there exists 
a point on $\gamma_n$ that is mapped by $P_{c}$ to $c(0)$.    We choose our parameterization of $\gamma_n$, equivalently of $\tilde{\gamma}_n$, so that that point is $\gamma_n(0)$.    Thus, $0\in I_n=[s_n,t_n]$, so $s_n<0<t_n$ and 
\[
						\gamma_n(s_n)=c(-n),\qquad P_{c}(\gamma_n(0))=c(0),\qquad \gamma_n(t_n)=c(n).
\]
By Theorem \ref{T:minimal implies QG} $\gamma_n$ is an $(A,a)$-quasi-geodesic, for some $A\geq 1, a\geq 0$ depending only on $S, g, T$.   
So by Morse's Lemma \ref{T:Morse}
\begin{equation}\label{E:proof shadow surjective 0}
							d_{H}\big(\gamma_n([s_n,t_n]),c([-n,n])\big)\leq R. 
\end{equation}
Hence the sequence $\gamma_n(0)$ lies in the closed ball of radius $R$ about $c(0)$.  Therefore, since the fibers of $S$ are compact, the sequence 
$\tilde{\gamma}_n(0)\in\Stilde$ lies in a compact subset of $\Stilde$, and so for a subsequence we have $\lim_{k\to\infty}\tilde{\gamma}_{n_k}(0)= z_*\in\Stilde$ 
whose projection down to $\Mtilde$ is $\pi(z_*)=q_*\in \overline{B}_{R}(c(0))$.    
Let $\tilde{\gamma}:\R\to\Stilde$ denote the Reeb trajectory with $\tilde{\gamma}(0)=z_*$.    We claim that for this trajectory \eqref{D:shadow 2} holds, and with this, the Lemma will be proven.  First we show that 
\begin{equation}\label{E:sn tn diverge}
		\lim_{n\to\infty}t_{n}=+\infty, \qquad \lim_{n\to\infty}s_{n}=-\infty.
\end{equation}
Since $\gamma_n:[s_n,t_n]\to\Mtilde$ is an $(A,a)$-quasi-geodesic $d_{\tilde{g}}\big(\gamma_n(s),\gamma_n(t)\big) \leq A|s-t|+a$ for all $s,t$ in the domain of $\gamma_n$. 
In particular 
\begin{equation}\label{E:proof shadow surjective 1}
			d_{\tilde{g}}\big(\gamma_n(0),\gamma_n(t_n)\big)\leq A|t_n|+a.
\end{equation}
Since the curvature is negative the geodesic $c$ is minimal and therefore $d_{\tilde{g}}\big(c(0),c(n)\big)\to\infty$.   
Thus the left hand side of \eqref{E:proof shadow surjective 1} diverges to $+\infty$ as $n\to\infty$ because $\gamma_n(0)$ remains in the compact 
set $\overline{B}_{R}(c(0))$, while $\gamma_n(t_n)=c(n)$.  Hence from \eqref{E:proof shadow surjective 1} $|t_n|\to\infty$, and therefore $t_n\to\infty$ as claimed.   The quasi-geodesic property of $\gamma_n$ also gives $d_{\tilde{g}}\big(\gamma_n(s_n),\gamma_n(0)\big)\leq A|s_n|+a$.   Similarly the left hand side diverges and so $|s_n|\to\infty$, and therefore $s_n\to-\infty$.   So \eqref{E:sn tn diverge} is proven.

Now we consider $\gamma:=\pi\circ\tilde{\gamma}:\R\to\Mtilde$.   It remains to show that $d_{H}\big(\gamma(\R),c(\R)\big)\leq R$, that is, that 
\begin{equation}\label{E:proof shadow surjective 2}
		\gamma(\R)\subset N_R(c)\qquad \mbox{and}\qquad c(\R)\subset N_R(\gamma)
\end{equation}
where $N_R(c), N_R(\gamma)$ denote the $R$-neighborhood of $c(\R)$ and $\gamma(\R)$ respectively.  
To show the first inclusion in \eqref{E:proof shadow surjective 2}: Let $\tau\in\R$.   We will show $\gamma(\tau)\in N_R(c)$.   
By \eqref{E:sn tn diverge} we have $\tau\in[s_n,t_n]$ for all $n\in\N$ sufficiently large, and so using \eqref{E:proof shadow surjective 0} 
$\gamma_{n}(\tau)\in\gamma_n([s_n,t_n])\subset N_R(c)$ for all $n$ sufficiently large.  
Thus $\gamma(\tau)=\lim_{k\to\infty}\gamma_{n_k}(\tau)$ must also lie in the closed set $N_R(c)$.   
To show the second inclusion in \eqref{E:proof shadow surjective 2}: Fix $\tau\in\R$.   We will show $c(\tau)\in N_R(\gamma)$.   
For $n\in\N$ sufficiently large $\tau\in[-n,n]$, so by \eqref{E:proof shadow surjective 0} there exists $\tau_n\in[s_n,t_n]$ so that 
\[
					\gamma_n(\tau_n)\in \overline{B}_{R}(c(\tau)).
\]
In particular the sequence $\gamma_n(\tau_n)$ is bounded.  But $\gamma_n$ is an $(A,a)$-quasi-geodesic shows that $|\tau_n|$ is bounded in $n$, so 
there exists a convergent subsequence $\tau_{n_j}\to\tau_*\in\R$.   Without loss of generality this is a subsequence of the subsequence $n_k$ 
chosen above, for which $\gamma_{n_k}(0)$ converges to $z_*$.    Thus, $\gamma(\tau_*)=\lim_{j\to\infty}\gamma_{n_j}(\tau_*)=\lim_{j\to\infty}\gamma_{n_j}(\tau_{n_j})$ 
where the final equality is proven using the $(A,a)$-quasi-geodesic property for each $\gamma_{n_j}$.   It follows that $\gamma(\tau_*)$ lies in the 
closed set $\overline{B}_{R}(c(\tau))$, and therefore $c(\tau)\in N_R(\gamma)$ as required.  This shows the second inclusion in \eqref{E:proof shadow surjective 2} and thus that the shadowing map is surjective.  
\end{proof} 

\begin{remark}\label{R:f equivariant}
The shadowing map $\tilde{f}$ is equivariant under deck transformations, and therefore descends to a map 
$f$ from (equivalence classes of) Reeb trajectories in $S$ to (equivalence classes of) geodesics in $M$.   
To see this it is conceptually and notationally simpler to consider $\tilde{f}$ as a map to trajectories of the co-geodesic flow of $\tilde{g}$ in the cotangent bundle (with say fixed energy $1/2$) rather than to the tangent bundle.   This way both $[\tilde{\gamma}]$ and $\tilde{f}([\tilde{\gamma}])$ can be considered as objects in the same space.   This is convenient for two reasons; firstly it means that the equivariance we wish to show is with respect to the deck transformation group for the covering $T^*\Mtilde\to T^*M$ (without having to distinguish this from $T\Mtilde\to TM$); secondly it allows to characterize $\tilde{f}([\tilde{\gamma}])$ simply by the condition 
\begin{equation}\label{D:f}
            d_{H}\big([\tilde{\gamma}],\tilde{f}([\tilde{\gamma}])\big)<\infty
\end{equation}
without projecting down to $\Mtilde$, slightly abusing notation by identifying $[\tilde{\gamma}]$ with the image of $\tilde{\gamma}$.  With this caveat, the 
equivariance we wish to show is simply that 
\begin{equation}\label{E:f equivariant}
        \tilde{f}\circ\Delta=\Delta\circ\tilde{f}  
\end{equation}
for each deck transformation $\Delta$ for the covering $\pi:T^*\Mtilde\to T^*M$ - writing $\Delta$ also for its action on equivalence classes, i.e. $\Delta([\tilde{\gamma}])=[\Delta\circ\tilde{\gamma}]$.    
To prove \eqref{E:f equivariant} we fix a Reeb trajectory $\tilde{\gamma}$ in $\Stilde$ and compare both sides on $[\tilde{\gamma}]$.  By \eqref{D:f} $\tilde{f}([\Delta\circ\tilde{\gamma}])$
is characterized by 
\begin{equation}\label{E: equivariance 100}
            d_{H}\big(\Delta([\tilde{\gamma}]),\tilde{f}([\Delta\circ\tilde{\gamma}])\big)<\infty.  
\end{equation}
On the other hand, $\Delta$ is automatically an isometry with 
respect to the metric on $T^*\Mtilde$ lifted from $T^*M$, and so from \eqref{D:f} we have $d_{H}(\Delta([\tilde{\gamma}]),\Delta(\tilde{f}([\tilde{\gamma}])))<\infty$.  Comparing with \eqref{E: equivariance 100} we see that $\tilde{f}([\Delta\circ\tilde{\gamma}])=\Delta(\tilde{f}([\tilde{\gamma}]))$ by uniqueness.  This proves \eqref{E:f equivariant} at $[\tilde{\gamma}]$ and hence equivariance of $\tilde{f}$.  
\end{remark}

\begin{remark}\label{R:f descends to injective}
Note that the quotient map associated to the shadowing map, see Remark \ref{R:f equivariant}, is surjective respectively injective, when $f$ itself is surjective respectively injective.   The only part of this claim that is not completely immediate is 
that $\tilde{f}$ injective implies $f$ is injective.  However, consider two Reeb trajectories 
in $S$ with distinct images, and let $\tilde{\gamma}_1$ and $\tilde{\gamma}_2$ be any lifts to $\Stilde$.  
Clearly the lifts also have distinct images, i.e.\ $[\tilde{\gamma}_1]\neq[\tilde{\gamma}_2]$.   Since each lift can be considered a lift for the covering space $T^*\Mtilde\to T^*M$, we have (abusing notation somewhat) 
$\Delta([\tilde{\gamma}_1])\neq[\tilde{\gamma}_2]$ for some deck transformation $\Delta$ of $T^*\Mtilde\to T^*M$.   So $\tilde{f}$ injective implies that 
$\tilde{f}\big(\Delta([\tilde{\gamma}_1])\big)\neq \tilde{f}([\tilde{\gamma}_2])$, i.e.\ by \eqref{E:f equivariant} $\Delta\big(\tilde{f}([\tilde{\gamma}_1])\big)\neq \tilde{f}([\tilde{\gamma}_2])$.  That is, $\tilde{f}([\tilde{\gamma}_1])$ and $\tilde{f}([\tilde{\gamma}_2])$ lie in different orbits under the deck group action on $T^*\Mtilde\to T^*M$, and so the curves they project to in $T^*M$ must have 
distinct images, that is, $f([\gamma_1])\neq f([\gamma_2])$.   Thus $f$ is also injective.  
\end{remark}

\section{Proof of main result}
Now we can prove the main result of this article, stated as Theorem \ref{thm:main 1 intro} in the introduction.   Here is the full statement.   Note that $S_{g}M\subset TM$ refers to the unit sphere bundle over $M$ 
with respect to the metric $g$: 

\begin{theorem}\label{T:main}
Suppose $M$ is a closed smooth manifold that supports a Riemannian metric $g$ of negative sectional curvature, and let $S\subset T^*M$ be a fiberwise star-shaped hypersurface.  Suppose that the Reeb flow on $S$ is coarsely minimal.    Then there is an orbital semi-conjugacy between the Reeb flow 
$\phi^t_S:S\to S$ and the geodesic flow $\phi^t_g:S_{g}M\to S_{g}M$.   That is, there exists a continuous surjective map 
\[
		F:S\to S_{g}M
\] 
and a continuous $T:\R\times S\to\R$, with $t\mapsto T(t,z)$ strictly monotonic increasing for each $z\in S$,  so that 
\begin{equation}\label{E:orbital conjugacy}
				\quad\qquad \phi^{T(t,z)}_{g}\big(F(z)\big)=F\big(\phi^{t}_S(z)\big)\qquad \forall z\in S.
\end{equation}
Moreover, $F$ is a homeomorphism if and only if the Reeb flow $\phi_S$ is divergent.  
\end{theorem}
\begin{proof}
Both footprint projection maps $S\to M$ and $\Stilde\to\Mtilde$ will be denoted by $\pi$.  In each case it should be clear from the context which is meant.  

  We will construct a continuous map $\tilde{F}:\Stilde\to S^1_{\tilde{g}}\Mtilde$ which will be equivariant with respect to deck transformations, and injective or surjective when the shadowing map $\tilde{f}$ is injective respectively surjective.   Hence, it will descend to a continuous map $F:S\to S^1_gM$. 
  We break the proof into 7 steps; the idea of using the additive co-cycle property of $\T$ in Step 4 is from 
  Gromov \cite{gromov2000three}, see also Knieper's more detailed presentation \cite{knieper2002hyperbolic}.    

\bigskip

\textsc{Step 1:} In this step we define a map $G:\Stilde\to S_{\tilde{g}}\Mtilde$ that has many of the properties we desire in $\tilde{F}$. 

\medskip

    Let $z\in\tilde{S}$, and $\tilde{\gamma}_{z}:\R\to\Stilde$ be the Reeb trajectory with initial condition $z$.   Denote by 
    $\gamma_z:=\pi\circ\tilde{\gamma}_z:\R\to\Mtilde$ the projection down to $\Mtilde$.   
    By assumption the Reeb flow is $T$-coarsely minimal for some $T\geq 0$, and so by Theorem \ref{T:minimal implies QG} 
    there are constants $A\geq 1, a\geq 0$ depending only on $g, T, S$, so that $\gamma_z$ is an $(A,a)$-quasi-geodesic.   
    By Morse's Lemma there is a unique (up to orientation preserving reparameterization) unit speed geodesic $c=c_z:\R\to\Mtilde$ that has finite Hausdorff distance from $\gamma_z$.   
    In terms of the shadowing map
    \[
    		\tilde{f}\big([\tilde{\gamma}_z]\big)=[\tilde{c}_z]  
    \]
    where $\tilde{c}_z=(c_z,\dot{c}_z)$ denotes the geodesic trajectory in $S_{\tilde{g}}\Mtilde$ that projects down to $c_z$.      
    We moreover choose $c_z$ to be the unique unit speed parameterization with $c_z(0)=P_{c}(\pi(z))$, that is, with initial condition the closest point 
    on $c(\R)$ to the footprint of $z$.   Now we define the map 
    \[
    			G:\Stilde\to S_{\tilde{g}}\Mtilde,\qquad z\mapsto (c_{z}(0),\dot{c}_{z}(0)).
    \]
    \medskip 
    
    \textsc{Step 2:} In this step we show that $G$ is continuous, that it maps the image of the Reeb trajectory $\tilde{\gamma}$ into the image of the geodesic trajectory $\tilde{f}([\tilde{\gamma}])$ shadowing it, and that $G$ is surjective.   

    \medskip 
  
    Let us first see why $G$ is continuous at a point $\tilde{z}_0\in\Stilde$. Let $R$ be the constant given by the Morse Lemma associated to the curvature of $g$ and the constants $A,a$.   Let $\tilde{\gamma}_0:\R\to\Stilde$ be the Reeb trajectory with initial condition $\tilde{\gamma}_0(0)=\tilde{z}_0$, and denote the projected curve by 
    $\gamma_0:=\pi\circ\tilde{\gamma}_0:\R\to\Mtilde$.   So $\gamma_0$ is an $(A,a)$-quasi-geodesic with initial condition 
    $\gamma_0(0)=z_0:=\pi(\tilde{z}_0)$.  So there exist distinct points at infinity $x,y\in\partial\Mtilde$ so that 
    \[
    \lim_{t\to-\infty}\gamma_0(t)=x, \qquad \lim_{t\to\infty}\gamma_0(t)=y.
    \]
    By Morse's Lemma there exists a shadowing geodesic $c_0:\R\to\Mtilde$, meaning $d_{H}(c_0(\R),\gamma_0(\R))\leq R$.  
    The image of $c_0$ is unique, and we choose $c_0$ to denote the unique unit speed parameterization with initial condition the closest point to $z_0$, that is, $c_0(0)=P_{c_0}(z_0)$, and orientation aligned with $\gamma_0$, that is $\lim_{t\to-\infty}c_0(t)=x$, and $\lim_{t\to\infty}c_0(t)=y$.  By definition $G(\tilde{z}_0)=(c_{0}(0),\dot{c}_{0}(0))$.  
    
    Consider a sequence $\tilde{z}_k\in\Stilde$ converging to $\tilde{z}_0$.    Let $z_k=\pi(\tilde{z}_k)\in\Mtilde$ be the projected sequence, 
    so $z_k\to z_0$, let $\tilde{\gamma}_k:\R\to\Stilde$ be the Reeb trajectory with initial condition $\tilde{z}_k$, and $\gamma_k:=\pi\circ\tilde{\gamma}_k:\R\to\Mtilde$ be the projected curve.   Again, each 
    $\gamma_k$ is an $(A,a)$-quasi-geodesic, and $\gamma_k\to\gamma_0$ converges 
    uniformly on compact subsets, that is, in $C^{\infty}_{\mathrm{loc}}(\R,\Mtilde)$.   Thus we find sequences $s'_k\to-\infty$, $t'_k\to\infty$, so that 
    \begin{equation}\label{E:x_ and y_k}
                    x_k:=\gamma_k(s'_k)\to x,\qquad y_k:=\gamma_k(t'_k)\to y
    \end{equation}
    as $k\to\infty$.    By Morse, each $\gamma_k$ has a shadowing geodesic $c_k:\R\to\Mtilde$, meaning $d_{H}(c_k(\R),\gamma_k(\R))\leq R$.  We choose a unit speed parameterization so that $c_k(0)=P_{c_k}(z_0)$.   From the shadowing property there exist $s_k, t_k\in\R$ so that 
    $d(c_k(s_k),x_k)\leq R$ and $d(c_k(t_k),y_k)\leq R$.   Thus, by \eqref{E:x_ and y_k} and the topology on $\overline{M}$, we have $c_k(s_k)\to x$ and $c_k(t_k)\to y$
    as $k\to\infty$.   Thus by Lemma \ref{L:basic convergence 1} 
    \begin{equation}\label{E:convergence alpha_k}
        c_k(\cdot\,+r_k)\to c_0
    \end{equation}
    in $C^{\infty}_{\mathrm{loc}}(\R,\Mtilde)$ for a suitable sequence of shifts $r_k\in\R$.   It follows, using also that each $c_k$ is a minimal geodesic, that the closest point on $c_k(\R)$ to $z_0$ converges as a sequence to the closest point on $c(\R)$ to $z_0$, that is $P_{c_{k}}(z_0)\to P_{c_0}(z_0)$.   In other words 
    $c_k(0)\to c_0(0)$.    On the other hand, from \eqref{E:convergence alpha_k}, we also have $c_k(c_k)\to c_0(0)$.   Thus 
    $|c_k|=d(c_k(c_k),c_k(0))\to 0$, and so \eqref{E:convergence alpha_k} implies $c_k\to c_0$ in $C^{\infty}_{\mathrm{loc}}(\R,\Mtilde)$, and thus $G(\tilde{z}_k)\to G(\tilde{z}_0)$.  This proves that $G$ is continuous. 
    
   From the definition of $G$, it maps each point on a Reeb trajectory $\tilde{\gamma}$ to a point of the form $\tilde{c}(t)=(c(t),\dot{c}(t))$ where 
   $[\tilde{c}]=f([\tilde{\gamma}])$.   In other words, $G$ maps $\tilde{\gamma}$ into the trajectory $\tilde{c}$ of the geodesic flow on 
   $S_{\tilde{g}}\Mtilde$ determined by the shadowing map.    We claim that $G$ maps $\tilde{\gamma}$ moreover surjectively onto $\tilde{c}$.   To see this, fix a parameterization $\tilde{\gamma}:\R\to\Stilde$, and consider $G(\tilde{\gamma}(t))$ for $|t|$ large.  
   Since $\gamma:\R\to\Mtilde$ is a quasi-geodesic we know that $\gamma(n)$ and $\gamma(-n)$ leave any given compact set if $n$ is 
   sufficiently large, and therefore so do the nearest points $P_{c}(\gamma(n))$ and $P_{c}(\gamma(-n))$ in $c(\R)$.  
   By continuity of $c^{-1}\circ P_{c}\circ\gamma:[-n,n]\to\R$ we see that the intermediate points $c(\R)$ are also covered, so $G$ maps onto the image of $\tilde{c}$.    Therefore by Proposition \ref{P:f onto}, 
    since $\tilde{f}$ is surjective, so is $G$ surjective.  

    \bigskip
    
    \textsc{Step 3:} In this step we show that $G$ is equivariant with respect to deck transformations.  

    \medskip

      The equivariance we mean is that if $\Delta$ is any deck transformation for the covering $T^*\Mtilde\to T^*M$ then 
     \begin{equation}\label{E:equivariance G}
      				\Delta\circ G=G\circ\Delta.  
      \end{equation}
      To show this, we compare both sides at a point $\tilde{\gamma}(0)$ for an arbitrary Reeb trajectory $\tilde{\gamma}:\R\to\Stilde$.   By definition of $G$, the defining characteristics of $G\big(\tilde{\gamma}(0)\big)$ are: First, 
      that it lies on the trajectory $\tilde{c}$ representing the shadowing trajectory $\tilde{f}([\tilde{\gamma}])$; secondly that on $\tilde{c}$ it projects down to $\Mtilde$ to the point that is closest to $\pi(\tilde{\gamma}(0))$, where $\pi:T^*\Mtilde\to\Mtilde$.   Applying $\Delta$ to these two conditions we see that $\Delta\big(G(\tilde{\gamma}(0))\big)$ lies on the trajectory $\Delta\circ\tilde{c}$ representing $\Delta\big(\tilde{f}([\tilde{\gamma}])\big)$, which by equivariance of $\tilde{f}$ - see Remark \ref{R:f equivariant} - is represented by $\tilde{f}([\Delta\circ\tilde{\gamma}])$, and moreover that within this curve the point $\Delta\big(G(\tilde{\gamma}(0))\big)$ projects down to $\tilde{M}$ to the point that is closest to $\pi(\Delta\circ\tilde{\gamma}(0))$ because $\pi\circ\Delta$ will be a deck transformation for the covering $\Mtilde\to M$, and hence an isometry of $\Mtilde$ as the metric is lifted from $M$.   In other words, $\Delta\big(G(\tilde{\gamma}(0))\big)$ satisfies the two defining characteristics of $G\big((\Delta\circ\tilde{\gamma})(0)\big)$.    We conclude that $\Delta\big(G(\tilde{\gamma}(0))\big)=G\big((\Delta\circ\tilde{\gamma})(0)\big)$.   This proves \eqref{E:equivariance G} at $\tilde{\gamma}(0)$ and hence in general.  

    \bigskip
    
    \textsc{Step 4:} In this step we define $\tilde{F}$ and show that it maps injectively along each Reeb trajectory.   

    \medskip 
    
    The problem with $G$ is that it is not necessarily injective, even along each Reeb trajectory.    We therefore modify $G$ along each Reeb trajectory, that is, to a map $\tilde{F}:\Stilde\to S_{g}\Mtilde$ of the form
    \begin{equation}\label{E:def of Ftilde} 
    							\tilde{F}(z)=\phi^{r(z)}_{\tilde{g}}\big(G(z)\big)\qquad \forall z\in\Stilde
    \end{equation}
    for some continuous function 
    \[
    						r:\Stilde\to\R.
    \]
    Our goal now is to choose $r$ appropriately, so that $\tilde{F}$ is injective on each Reeb trajectory (and also so that $\tilde{F}$ 
    is equivariant with respect to deck transformations, see Step 6).        
 Since $G$ maps Reeb trajectories to $\tilde{g}$-geodesic trajectories, there is a map 
   $\T:\R\times \Stilde\to \R$ satisfying 
   \begin{equation}\label{E:Def of T}
   		G\left(\phi_{\Stilde}^t(z)\right)=\phi^{\T(t,z)}_{\tilde{g}}\big(G(z)\big)\qquad \forall t\in\R,\ z\in\Stilde.
   \end{equation}
   Note that as $g$ has negative curvature there are no contractible closed geodesics on $M$ and therefore no closed trajectories of the geodesic flow on $S_{\tilde{g}}\Mtilde$, which means that $\T$ is uniquely characterized by \eqref{E:Def of T}.  In particular, $\T(0,\cdot)\equiv0$.  The continuity of $G$ and of the flows $\phi_{\Stilde}^t$, $\phi_{\tilde{g}}^t$ implies that $\T$ is also continuous.  Moreover, 
   $\T$ is an `additive co-cycle', that is 
   \begin{equation}\label{E:cocycle}
   		\T(s+t,z)=\T(s,{\phi}_{\tilde{S}}^t(z))+\T(t,z)\qquad \forall s,t\in\R,\ \forall z\in\Stilde. 
   \end{equation} 
   Indeed, this follows by applying $G$ to the relation 
   $\phi_{\tilde{g}}^{s+t}(z)=\phi_{\tilde{g}}^{s}\big(\phi_{\tilde{g}}^{t}(z)\big)$ and inserting this into \eqref{E:Def of T} and using uniqueness of $\T$.   Now fix $z\in\Stilde$ arbitrarily.    To show that $\tilde{F}$ is injective along the Reeb trajectory through $z$ means to show that $t\mapsto\tilde{F}(\phi_{\Stilde}^{t}(z))$ is injective.  From \eqref{E:def of Ftilde} and \eqref{E:Def of T} 
   \[
   		\tilde{F}(\phi_{\Stilde}^{t}(z))=\phi^{r(\phi_{\Stilde}^{t}(z))}_{\tilde{g}}\big(G\big(\phi_{\Stilde}^{t}(z)\big)=
						\phi^{r(\phi_{\Stilde}^{t}(z))}_{\tilde{g}}\Big(\phi^{\T(t,z)}_{\tilde{g}}(G(z))\Big)
   \]
   for all $t\in\R$, so that 
   \[
   		\tilde{F}(\phi_{\Stilde}^{t}(z))=\phi^{m(t,z)}_{\tilde{g}}\big(G(z)\big)
   \]
   where 
   \begin{equation}\label{E:def of m}
   		m(t,z):=r(\phi_{\Stilde}^{t}(z)) + \T(t,z). 
   \end{equation}
   Thus $\tilde{F}$ will be injective along each Reeb trajectory if we can choose the function $r$ so that for each $z\in\Stilde$ the function $t\mapsto m(t,z)$ is injective.   
   To do this we model Gromov's approach from the purely Riemannian situation in \cite{gromov2000three}: Using the quasi-geodesic estimates once more, there exists a universal 
   $\tau>0$ such that $\T(\tau,z)>0$ for all $z\in \tilde{S}$.   We will verify this in the next step below.   Now define $r:\tilde{S}\to \R$ and $\tilde{F}:\tilde{S}\to S_{\tilde{g}}\tilde{M}$ by
   \begin{equation}\label{E:def of r}
   				r(z):=\frac{1}{\tau}\int^\tau_0\T(t,z)\,dt. 
  \end{equation}
   Then we compute 
   \begin{align}
   			m(t,z)&=\frac{1}{\tau}\int^\tau_0 \T(s,\phi_{\Stilde}^{t}(z))\,ds\ +\ \T(t,z)\nonumber\\
				&=\frac{1}{\tau}\int^\tau_0 \T(s,\phi_{\Stilde}^{t}(z)) + \T(t,z)\,ds\nonumber\\
				&=\frac{1}{\tau}\int^\tau_0 \T(s+t,z)\,ds=\frac{1}{\tau}\int^{\tau+t}_{t} \T(s,z)\,ds,\nonumber
   \end{align}
   so by continuity of $\T$ the partial derivative $\partial_tm(t,z)$ exists and
   \begin{equation}\label{E:derivative of m}
   			\partial_tm(t,z)=\frac{1}{\tau}\Big(\T(\tau+t,z)-\T(t,z)\Big)=\frac{1}{\tau}\T(\tau,{\phi}_{\tilde{S}}^t(z))>0  
   \end{equation}
   where the final term is positive by choice of $\tau$.  This strict monotonicity of $m$ implies that $\tilde{F}$ maps each flow line  injectively into a flow line.   
   
   \bigskip
    
    \textsc{Step 5:} In this step we show that there exists $\tau\in\R$ so that $\T(\tau,\cdot)>0$, and also that $\tilde{F}$ is surjective.

    \medskip 

    Let $z\in\Stilde$  Denote by $\tilde{\gamma}:\R\to\Stilde$ the Reeb trajectory with initial condition $\tilde{\gamma}(0)=z$ and let $\tilde{c}:\R\to S_{\tilde{g}}\Mtilde$ be the $\tilde{g}$-geodesic trajectory with initial condition $\tilde{c}(0)=G(z)$.   By \eqref{E:Def of T} $G(\tilde{\gamma}(t))=\tilde{c}(\T(t,z))$ for each $t\in\R$, so, writing $\gamma$ and $c$ for the projected curves in $\Mtilde$, we have $\gamma(t)$ is the closest point on $\gamma(\R)$ to $c(\T(t,z))$, and hence a distance at most $R$.   Thus  
    \begin{align}
        \T(t,z) = d_{\tilde{g}}\big(c(\T(t,z)),c(0)\big)  
                \geq d_{\tilde{g}}\big(\gamma(t),\gamma(0)\big) - 2R 
                \geq\frac{1}{A}|t|-a - 2R\nonumber
    \end{align}
   because $\gamma$ is an $(A,a)$-quasi-geodesic.    So there exists $\tau>0$ sufficiently large, independent of $z$, so that $\T(t,z)/t\geq 1/(2A)$ whenever $t\geq \tau$.    In particular $\T(\tau,\cdot)$ is uniformly bounded from below by a positive value.  It also follows that if we use this $\tau$ to define $r$ in \eqref{E:def of r}, then by \eqref{E:derivative of m} the resulting function $m$ will satisfy $\partial_tm(t,z)\geq 1/(2A)$ and so $\tilde{F}$ maps the trajectory $\tilde{\gamma}$ surjectively onto $\tilde{c}$.    It follows that $\tilde{F}$ is surjective, since by Proposition \ref{P:f onto} the shadowing map is surjective.  

    \bigskip
    
    \textsc{Step 6:} In this step we show that $\tilde{F}$ is equivariant with respect to deck transformations, so that it descends to a 
      continuous map $F:S\to S_{g}M$.   

\medskip

      Let $\Delta$ be a deck transformation for the covering $T^*\Mtilde\to T^*M$.  Then we wish to show that 
      \begin{equation}\label{E:equivariance F}
      				\Delta\circ\tilde{F}=\tilde{F}\circ\Delta.   
      \end{equation}
      From the representation for $\tilde{F}$ in \eqref{E:def of Ftilde} and the equivariance for $G$ proven above, it suffices to check that $r(z)=r(\Delta(z))$ for each $z\in\Stilde$.   From \eqref{E:def of r} this follows if $\T$ is invariant in the sense that $\T(t,\Delta(z))=\T(t,z)$ for all $(t,z)\in\R\times\Stilde$, which in turn follows from \eqref{E:Def of T} using the equivariance of $G$.  
      
      Note similarly that from \eqref{E:def of m} $m$ is invariant 
     so that the right hand side of \eqref{E:T descends} is invariant as claimed earlier.  
   
    \bigskip
    
    \textsc{Step 7:} In this step we verify that the quotient map $F$ has all the desired properties.  

    \medskip 
    
    By Step 6 we have that $\tilde{F}$ descends to a continuous map $F:S\to S_gM$, and that it is surjective because $\tilde{F}$ is surjective.   
    We also see that $F$ maps Reeb orbits to geodesics and that \eqref{E:orbital conjugacy} holds where $T$ is 
    defined by 
  \begin{equation}\label{E:T descends}
   T(t,\pi(z)):=m(t,z)-r(z) \qquad \forall (t,z)\in\R\times\Stilde.  
 \end{equation}
 Note that the right hand side is invariant under deck transformations, see end of Step 6, so that this relation 
 makes $T$ well defined.     
 It also follows that for each fixed $w=\pi(z)\in S$ the function $t\mapsto T(t,w)$ is differentiable with derivative 
 $\partial_tT(t,\pi(z))=\frac{1}{\tau}\T(\tau,{\phi}_{\tilde{S}}^t(z))>0$, hence $T$ is strictly monotonic increasing in the $t$ variable as claimed in the theorem, and therefore $F$ is 
 injective along each Reeb trajectory.      Now if the Reeb flow is divergent then the shadowing map, and hence $\tilde{F}$, maps distinct Reeb trajectories to distinct geodesic trajectories, and therefore by Step 4 $\tilde{F}$ is injective also.   From the equivariance for $\tilde{F}$ it follows that $F$ is also injective, arguing analogously to in Remark \ref{R:f descends to injective}.    So if the Reeb flow is divergent then $F$ is a continuous bijection between compact and Hausdorff spaces, so it is a homeomorphism.   
 Conversely, if $F$ is a bijection then it maps distinct Reeb trajectories to distinct geodesic trajectories, and therefore the same holds for $\tilde{F}$, and so the shadowing map $\tilde{f}$
 is also injective, which as observed in the proof of Proposition \ref{P:f onto} is equivalent to the Reeb flow being divergent.   This completes the proof of Theorem \ref{T:main}.  
\end{proof}

\section{An example with non-convex fibers}\label{S:example}
In this section we prove Theorem \ref{thm:example intro} from the introduction, which shows that the assumptions in our main result do not force the fibers of the star-shaped hypersurface $S\subset T^*M$ to be convex.   We recall the statement: 

\begin{theorem}\label{thm:nonconvexhypersurface}
    Let $M$ be a closed smooth manifold which supports a Riemannian metric $g$ of negative sectional curvature. Then there exists a star-shaped hypersurface $S\subset T^*M$ which is coarsely minimal and has the divergence property, which is moreover not fiber-wise convex. That is, there are fibers $S_x$ that are strictly not convex.
\end{theorem}

To emphasize: If the Reeb flow on $S$ is minimal in the stronger sense that each chord has minimal action amongst Reeb chords with the same end fibers and in the same homotopy class, then this can be shown to force the fibers to be convex, see the Appendix or \cite{dePooter_thesis}, which essentially restricts the discussion to Finsler geodesic flows.   

\subsection{Variational theory and convexity}
The following statements can all be found in \cite{fathi2008weak}*{Sections 2 and 3}. Statements on lower regularity can also be found there.

Let $L:TM\to \R$ be a $C^\infty$ map. Choosing a bundle chart with coordinates $(x_i,v_i)_{i=1,\ldots,n=\dim M}$, we define $D^2_VL(x,v)$ to be the Hessian matrix of $L_x:T_xM\to \R$ at $v\in T_xM$. I.e. the matrix consisting of all second-order partial derivatives with respect to the fiber coordinates. We say that $L$ is \textbf{non-degenerate} at $(x,v)$ if $D^2_V(x,v)$ is positive definite. It is readily verified that this property is independent of bundle chart (but not independent of general chart). The map $L$ is said to be \textbf{super-linear} if, for any Riemannian metric $g$ on $M$, 
$$\lim_{\|v\|\to \infty}\|L(x,v)\|/\|v\|=\infty\quad\forall(x,v)\in TM.$$
Here the norms are with respect to $g$. Since any two Riemannian metrics are equivalent as norms on each fiber $T_xM$, this property is independent of chosen metric. 

Replacing $TM$ by $T^*M$ instead, we similarly define the vertical Hessian $D^2_VH(x,p)$ of a smooth map $H:T^*M\to \R$ and say $H$ is non-degenerate at $(x,p)$ when the Hessian is positive definite. Additionally it is super-linear if the above limit exists and is infinite everywhere for any chosen co-metric on $T^*M$. 
\begin{definition}[Tonelli Lagrangian/Hamiltonian]
    Let $L:TM\to \R$ and $H:T^*M\to \R$ be smooth maps. The Lagrangian $L$ and Hamiltonian $H$ are \textit{Tonelli} if they are non-degenerate at every point and super-linear, respectively.\qet
\end{definition}
Tonelli Hamiltonians and Lagrangians are dual to eachother via the Fenchel transform, and their dynamics are conjugated by the Legendre transform. Recall the Fenchel transform $H:=L_F$ of a fiber-wise convex Lagrangian $L$ by
$$L_F(x,p):=\sup_{v\in T_xM}\left(pv-L(x,v)\right).$$
For a fiber-wise convex Hamiltonian $H$ its Fenchel transform $L:=H_F$ is defined similarly, replacing $v\in T_xM$ by $p\in T_x^*M$ and $L$ by $H$. After identifying the double-dual of $T_xM$ with itself, it holds that $(L_F)_F=L$ whenever the supremum is attained in each fiber (e.g. when $L$ is super-linear).

In general, if $L$ or $H$ is continuous, their Fenchel transforms are also continuous, but higher regularity does not imply higher regularity of their transforms. However, if $L$ or $H$ is Tonelli, then their transforms are also Tonelli. Moreover, if either is additionally smooth, their transforms are smooth, too.

The Legendre transform $\mathcal{L}:TM\to T^*M$ of a $C^1$ Lagrangian is defined as
$$(x,v)\mapsto \left(x, \frac{\partial L}{\partial v}(x,v)\right).$$
It is a fibered map, which sends $(x,v)$ to a value $(x,p)$ for which $v$ is a maximizer in the above supremum. The supremum may in general be obtained at multiple points, but if $L$ or $H$ is strictly-convex (i.e. sublevel sets are fiber-wise strictly-convex) it is unique. It can then be shown that $\cal{L}$ is a homeomorphism. If $L$ is additionally Tonelli and smooth, then the Legendre transform is a smooth diffeomorphism too. For a Hamiltonian $H$, the identification of the double-dual of $T_xM$ with itself implies an analogous definition of the Legendre transform $\cal{H}:T^*M\to TM$, and if $\cal{L}$ is bijective, then the Legendre transform of $L_F$ is the inverse of $\cal{L}$.

Another important result of Tonelli Lagrangians is that critical values of the Lagriangian energy functional are integral curves of the Euler-Lagrange vector field. That is, for a Lagrangian $L$ and an absolutely continuous curve $\gamma:[0,1]\to M$, the Lagrangian energy of $\gamma$ is defined as
$$\mathds{L}(\gamma):=\int^1_0L(\gamma(t),\dot{\gamma}(t))dt.$$
All absolutely continuous curves with equal starting and end-points can be given the structure of a Banach manifold, and the Euler-Lagrange Theorem shows that critical points (curves) of the Lagrangian energy solve the Euler-Lagrange equation, so long as $L$ is at least $C^1$. Conversely, solutions of the Euler-Lagrange equation are also critical points. Tonelli's Theorem states: if $L$ is additionally fiber-wise convex and super-linear, then a minimizer does indeed exist. If $L$ is moreover Tonelli and smooth, then there is a globally well-defined smooth vector-field $X_L$ on $TM$, called the Euler-Lagrange vector field, whose integral curves are all solutions of the Euler-Lagrange equation.

Finally, if $L$ is a Tonelli Lagrangian, the Legendre transform intertwines the Euler-Lagrange flow and Hamiltonian flow of its Fenchel transform. The $\mathds{L}$-energy of an absolutely continuous curve and action $\mathds{A}$ of $L_F$ are additionally equal. That is,
$$\phi_H^t\circ \cal{L}=\cal{L}\circ \phi^t_L\quad\mathrm{and}\quad \mathds{L}(\gamma)=\mathds{A}(\mathcal{L}\circ\gamma).$$
Indeed, 
$$\mathds{A}(\mathcal{L}\circ\gamma)=\int^1_0\frac{\partial L}{\partial v}(\gamma(t),\dot{\gamma}(t))(\dot{\gamma}(t))-L_F(\gamma(t),\dot{\gamma}(t))dt=\int^1_0L(\gamma(t),\dot{\gamma}(t))dt=\mathds{L}(\gamma).$$
The same holds for the Legendre-Fenchel transform of a Tonelli Hamiltonian $H$.

By an approximation argument, we find the following corollary (which is not in \cite{fathi2008weak}):
\begin{corollary}\label{cor:fathicor}
    Let $H:T^*M\to \R$ be a smooth function, fiber-wise convex Hamiltonian with compact sub-level sets, where is $M$ a closed smooth manifold. Let $L:TM\to \R$ be its Fenchel transform. Then there is a smooth curve $\gamma:[0,1]\to M$ which is the projection of a Hamiltonian chord with least $\mathds{L}$-energy.
\end{corollary}
\begin{proof}
    Let $l:=\inf \{\mathds{L}(\gamma):\gamma\in C^{ab}([0,1],M)\}\geq 0$ be the infimum of $\mathds{L}$-energy, and let $\gamma_n$ be an infimizing sequence.
    
    Let $g$ be a smooth Riemannian metric on $M$, with $g^*$ the cometric. For any $\epsilon>0$, $H_{\epsilon}:=H+\epsilon g^*$ is a smooth Tonelli Hamiltonian. Let $L_{\epsilon}$ denote the Fenchel transform. Let $\{\epsilon_n>0\}_{n\in \N}$ be sequence converging to zero, and let $\{\alpha_n:[0,1]\to T^*M\}_{n\in \N}$ be a sequence of minimizing integral curves, i.e. their projections $\beta_n=\pi\circ\alpha_n$ are minimizers of $\mathds{L}_{\epsilon_n}$. As sub-level sets of $H$ are compact, we assume without loss of generality that $\alpha_n$ converges to some integral curve $\alpha$ of $X_H$. Denote by $\beta=\pi\circ \alpha$ the projection of the curve. Since the Fenchel map (assigning a Hamiltonian to its dual Lagrangian) is continuous, $\mathds{L}(\beta)=\lim\:\mathds{L}_{\epsilon_n}(\beta_n)$.

    Assume now for contradiction that $\mathds{L}(\beta)>l$. Let $\tilde{n}\in \N$ be large enough so that $\mathds{L}(\gamma_{\tilde{n}})<\mathds{L}(\beta)$. Selecting a possibly larger number $\overline{n}\in \N$, this implies that $\mathds{L}_{\epsilon_{\overline{n}}}(\gamma_{\overline{n}})<\mathds{L}_{\epsilon_{\overline{n}}}(\beta_{\overline{n}})$, a contraction. This finishes the proof.
\end{proof}

\subsection{Construction}
In this section we construct the hypersurface $S$ of Theorem \ref{thm:nonconvexhypersurface}. Let $M$ be a closed smooth manifold which supports a Riemannian metric $g$ of negative curvature. Let $U\neq \varnothing$ be a geodesically convex neighborhood in $M$, i.e. any two points in $U$ are connected by a unique geodesic contained in $U$, which also minimizes length. Let $x,y,z$ be three distinct points in $U$, and let $\alpha_i:[0,T_i]\to U$ $(i=1,2,3)$ be the unit speed geodesics connecting the pairs $(x,y),(y,z),(x,z)$ respectively. We make use of the isotopy extension theorem for isotropic submanifolds \cite{geiges2008contact}*{Theorem 2.6.2}, applied to points:
\begin{corollary}[Corollary 2.6.3 of \cite{geiges2008contact}]
Let $(N,\xi)$ be a connected contact manifold and $p,q\in N$ two (not necessarily distinct) points of $N$. Let $\gamma:[0,1]\to N$ be a smooth path connecting $p=\gamma(0)$ with $q=\gamma(1)$. Then there is a compactly supported contact isotopy $\{\psi_t\}_{t\in[0,1]}$ of $N$ with $\psi_t(p)=\gamma(t)$. In particular there is a contactomorphism $\psi_1$ with $\psi_1(p)=q$.
\end{corollary}

Let $\gamma:[0,1]\to S^1_gU$ be a smooth path such that $\gamma(0)=\alpha_1'(1)$ for and $\gamma(1)=\alpha_2'(1)$, such that $\pi\circ\gamma(t)=\alpha_2(tT_2)$ for all $t\in[0,1]$. It should be clear that such a path exists. We apply the theorem within a small neighbourhood $V$ of $\gamma$ in $S^1_gM$ (which is again a contact manifold), so that its support has compact footprint in $U$. Let $\phi_t$ denote the contact isotopy given by the theorem, which naturally extends as the identity outside $V$ to $S^1_gM$ (and denote it with the same letter). Recall that the Liouville form is the chosen contact form on $S^1_gM$. As $\psi_1$ is a contact isotopy, there exists a smooth positive function $f:S^1_gM\to \R$ such that $(\psi_1^{-1})^*\lambda=f\lambda$, which is equal to one outside of $V$. 

To construct $S$, we will use $f$ to `rescale' $S_g^1M$: let $T^*M\setminus 0_M$ be the cotangent bundle without the zero section, and let $\phi:S^1M \to T^*M\setminus 0_M$ be the fibered map defined by $(x,p)\mapsto (x,f(x,p)\cdot p)$. Since $f$ is smooth, $\phi$ is smooth, and moreover $\phi^*\lambda=f\lambda$. Indeed, $\phi$ preserves fibers, so it follows from the chain rule that $d\pi\circ \phi_*=d\pi$. It is straightforward to verify that $\phi$ is an embedding, and that its image $S:=\phi(S^1_gM)$ is star-shaped and of contact-type. 

An important observation is that $(S^1_g,\lambda)$ and $(S,\lambda)$ are strictly contactomorphic via $\phi\circ \psi_1$, by construction. Note additionally that by construction this map is the identity outside of $S^1_gU$. 

It remains to prove the claims about $S$. Let $\tilde
S$ denote the universal cover of $S$ in $T^*\tilde{M}$, and observe that $\phi\circ\psi_1$ lifts to a strict contactomorphism $\nu:S^1_{\tilde{g}}\tilde{M}\to \tilde{S}$ which is the identity outside of all preimages of $S^1_gU$. As $U$ is geodesically convex and its closure is compact, there is a maximal time $T>0$ such that a unit speed geodesic $\alpha:\R\to \tilde{M}$ spends at most time $T$ in the connected pieces of the preimage of $U$. Putting the two together: for any Reeb chord $\gamma:\R\to \tilde{S}$ and real numbers $a<b\in \R$, there exist $a'\leq a$ and $b'\geq b$ with $|a-a'|\leq T$ and $|b-b'|\leq T$ such that $\alpha|[a',b']$ is a chord of minimal action between its end-points. Indeed, as $\nu$ is a strict contactomorphism and the identity outside of the preimage of $U$, it identifies $S^1_{\tilde{g}}\tilde{M}$ chords and $\tilde{S}$ connecting the base point $\alpha(a')$ and $\alpha(b')$ while preserving the action. As $\tilde{g}$ has non-positive curvature, all these chords have minimal action. In conclusion, $S$ is coarsely minimal.

Additionally, since $\phi\circ\psi_1$ has support in $S^1_g U$, Reeb chords of $\tilde{S}$ have Haudsdorff distance at most $T$ from $\tilde{g}$-chords. As $\tilde{g}$ has negative curvature, it has the divergence property. Hence, the Reeb chords $\tilde{S}$ must then also have the divergence property.

\subsubsection*{Non-convexity}
Finally, we show that $S$ has some non-convex fibers: let $\beta:\R\to S^1_gM$ denote the Reeb chord of $f\lambda$ starting at $\alpha_1'(0)$. By construction, $\beta(T_1)=\alpha_2'(1)$. We want to show that there is a point $z'$ near $z$ on $\mathrm{im}\: \pi\circ \beta$, such that the tangent vector of $\pi\circ \beta$ is not parallel to the tangent vector of the shortest geodesic from $x$ to $z'$. As we are able to choose the neighborhood $V$ as small as we like, and the image of $\alpha_2$ only intersects $\mathrm{im}\:\alpha_3$ at the end point, there must be such a point. Moreover, by a small enough choice of $V$, the geodesic chord connecting $x$ to $z'$ is also a $f\lambda$-chord.

So far we have produced a point $z'\in U$ which has at least two chords connecting $S^1_{g,x}M$ to $S^1_{g,z'}M$, whose projections of tangent vectors to $TM$ are not parallel. The set of such points is also open. Arguing as at the beginning of \ref{S:Def FH}, Sard's theorem implies that there is a point $z''$ which is additionally regular. This allows us to finally choose a point $\overline{z}$ in $U$ which has all three properties, and additionally has that the two chords $\beta_1,\beta_2:\R\to S^1_gM$ have unequal length with respect to parameterization. 

Assume for contradiction that $S$ is convex. We will denote by $\tilde{\beta}_i:\R\to S$ the composition $\phi\circ\beta$. Note that $\pi\circ \beta=\pi\circ \tilde{\beta}$. Without loss of generality, $\tilde{\beta}_1$ is shorter and $\pi\circ\beta_2(0)=x$. Let $t_1<0$ be the last time where $\pi\circ \beta_2$ intersects $\partial U$ before hitting $x$, and let $t_3>0$ be the first time thereafter it hits $\partial U$, and let $t_{2,i}$ be the first time for which $\pi\circ\beta_i(t)=x$. By construction, $\pi\circ \beta_2$ is the unique chord in its homotopy class connecting $\pi\circ \tilde{\beta}_2(t_1)$ to $\pi\circ \tilde{\beta}_2(t_2)$. Indeed: for the geodesic flows this is true (since it has negative curvature), and the strict contactomoprhism connecting preserves this, the map has compact support in $U$. 

However, if we switch to the Lagrangian perspective, we obtain the contradiction: define the absolutely continuous curve $\zeta:[0,t_3-t_1]\to M$ by following $\pi\circ \beta_2|[t_1,0]$, then $\pi \circ\beta_1(t_{2,1}/{t_{2,2}})|[0,t_{2,2}]$, and finally $\pi\circ\beta_2|[t_{2,2},t_3]$. Let $\tilde\zeta(t):=\zeta(t(t_3-t_1)):[0,1]\to M$ be the rescaled map: it has less $\mathds{L}$-energy than $\pi\circ\beta(t/(t_3-t_1))$, but it is not $C^1$. Hence, it cannot be an integral curve. However, there are no other integral curves connecting the end points, which contradicts Corollary \ref{cor:fathicor}. The proof of Theorem \ref{thm:nonconvexhypersurface} is now complete.   

\appendix
\section{Non-minimality of non-convex hypersurfaces}\label{S:Appendix}
In this Appendix we will show that replacing coarse minimality by minimality forces fibers to be convex, see Theorem \ref{thm:nonconvexunequalchord}. That is, consider a fiberwise star-shaped hypersurface $S\subset T^*M$ in the cotangent bundle of a smooth manifold $M$.  If the Reeb flow of $S$ has the property that each chord minimizes action in the sense of Definition \ref{D:minimal Reeb}, then we will show that each fiber $S_x=S\cap T^*_xM$ bounds a convex neighborhood of the origin. We include a proof as we could not find this in the literature.  

This result is entirely local, and Theorem \ref{thm:nonconvexunequalchord} proves the contrapositive statement: If $S_x=S\cap T^*_xM$ is a non-convex fiber, then in any neighbourhood $U$ of $x$ there exists a point $y\in U$ for which there exist two Reeb chords $\gamma_i:[0,T_i]\to S$ from $S_x$ to $S_y$ (in that order), whose projections to the base lie in $U$, are homotopic within $U$ with fixed ends, but have distinct terminal times $T_i$. Obviously, at least one of these chords is not minimal in the sense of Definition \ref{D:minimal Reeb}.  

To give an idea of the proof: We will first prove the statement for what we call the linearized flow, i.e.\  where there is a bundle chart of $T^*M$ for which the fiber is translation invariant. Then we show that this implies the full statement.  The proof of the linearized statement takes a detour through Morse theory for star-shaped bodies: In the next section we will prove Theorem \ref{thm:morseofnonconvex} and Corollary \ref{cor:morseofnonconvex}, which roughly say that there is a direction at $x$ for which at least two chords of the linearized flow in that direction have different speeds.

\subsection{Morse functions for non-convex star-shaped sets}
The main goal of this subsection is to prove Theorem \ref{thm:morseofnonconvex}, which states that a smooth boundary of non-convex star-shaped set in $\mathbb{R}^n$ (for $n\geq 2$) must have at least three critical points that live in distinct level sets.  

Let $S\subset \mathbb{R}^{k+1}$ be a closed oriented smooth embedded manifold. Let $q\in \mathbb{S}^k$ be a unit vector, and let $F_q:\mathbb{R}^{k+1}\to \R$ be the projection along $q$, given by $x\mapsto \langle q,x\rangle$, and let $f_q$ denote the restriction of $F_q$ to $S$.

Let $G:S\to \mathbb{S}^{k}$ be the Gauss map. For a point $x\in S$ there exists a smooth function $H:\mathbb{R}^{k+1}\to \R$ and an open neigbhourhood $U\subset \mathbb{R}^{k+1}$ containing $x$ such that zero is a regular value of $H$, $\nabla H$ points in the normal direction of $S$, and $H^{-1}(0)\cap U=S\cap U$. We will call such a pair $(H,U)$ a generating pair for $S$ at $x$. 

\begin{proposition}\label{prop:gaussandheight}
    Let $S\subset \mathbb{R}^{k+1}$ be a closed oriented smooth embedded manifold, and let $q\in\mathbb{S}^k$ be a unit vector. A point $x\in S$ is a critical point of $f_q$ if and only if $G(x)=\pm q$.
    
    Moreover, let $(H,U)$ be a generating pair for $S$ at $x$. Then the index, nullity, and co-index of $\nabla^2H(x)|_{TS}$ is respectively equal the the co-index, nullity, and index of $f_q$ at $x$.

    Finally, $x\in S$ is a regular critical point of $f_q$ if and only if it is a regular point of $G$, i.e.\ $dG(x)$ has full rank.
\end{proposition}
\begin{proof}
    Let $(H,U)$ be a generating pair for a point $x\in S$. Then observe that $G(y)=\nabla H(y)/\|\nabla H(y)\|$ in $U\cap S$. The method of Lagrange multipliers implies that $x$ is a critical value of $f_q$ if and only if $\nabla F_q(x)=\langle q,\cdot\rangle $ is a scalar multiple of $\nabla H(x)$. This proves the first statement.

    For the second statement: let $x$ be a critical point of $f_q$, and let $(H,U)$ be any generating pair. Implicitly, $S\cap U$ is defined by $H(y)=0$. Without loss of generality, assume that $q=(1,0,\ldots,0)\in\mathbb{S}^{k}$, so that locally we may choose coordinates $(x_1,\ldots,x_k)\in \mathbb{R}^{k}$ and a smooth function $f:\R^k\to \R$ satisfying $H(f(x_1,\ldots,x_k),x_1,\ldots,x_k)=0$. The implicit function theorem implies that $\nabla^2H(x)|_{TS}=-\nabla^2f/\|\nabla H\|$ (where the Hessian on the right hand side is taken in $\R^k$). This proves the second statement.

    For the final statement, observe that in the above coordinates $\nabla ^2f_q=\nabla^2f$, while $G(y)=\nabla H(y)/\|\nabla H(y)\|$ in $S\cap U$, so that $\nabla G(y)=\nabla^2H(y)/\|\nabla H(y)\||_{TS}$, which is proportional to $\nabla^2f(y)$ be the previous calculation. This proves the third claim.
\end{proof}

\begin{proposition}\label{prop:densemorseheight}
    For every closed embedded smooth manifold $S\subset \mathbb{R}^{k+1}$ there is an open and dense set $A\subset \mathbb{S}^k$ of full-measure, such that $f_q$ is a Morse function for all $q\in A$. 
\end{proposition}
\begin{proof}
    Let $q\in \mathbb{S}^k$ be a unit vector. By the above proposition, the critical points of $f_q$ are given by $G^{-1}(q)\cup G^{-1}(-q)$. The final claim of the proposition shows that the critical points are all regular if and only if they are regular points of $G$. By Sard's Lemma, there is an open and dense set $A'\subset \mathbb{S}^k$ of regular values of $G$, with full measure. Intersecting this set with minus itself will still be open, dense, and of full-measure. Hence, $A=A'\cap -A'$ is the desired set.
\end{proof}

Let $D\subset \R^{k+1}$ be a star-shaped domain with respect to zero, with a smooth boundary $S$ that is transverse to the radial vector field $R=\sum^{k+1}_{i=1}x_i\partial_i$. Then $S$ is a canonically oriented and embedded closed smooth manifold.

\begin{theorem}\label{thm:morseofnonconvex}
    Suppose that the star-shaped domain $D$ is not convex, then there exists $q\in \mathbb{S}^k$ such that the function $f_q:S\to \R$ is Morse and has at least three critical points with distinct values.
\end{theorem}

Before we prove this result, note the following corollary:
\begin{corollary}\label{cor:morseofnonconvex}
    Suppose that the star-shaped domain $D$ is not convex, and let $H_S:\R^{k+1}\to \R$ be the two-homogeneous function associated with $S=H^{-1}(1/2)$. Then there exists $q\in \mathbb{S}^{k+1}$ and at least two distinct points $p_0,p_1\in S$ with $\nabla H_S(p_0)=c\nabla H_S(p_1)$ for some real constant $c>0$ unequal to $1$.
\end{corollary}
\begin{proof}
    Let $q\in \mathbb{S}^k$ be unit vector given by Theorem \ref{thm:morseofnonconvex} such that $f_q:S\to \R$ is Morse with at least three distinct critical values. Recall from Proposition \ref{prop:gaussandheight} that critical points $x\in S$ of $f_q$ were precisely the points where the unit normal $G(x)$ was parallel or anti-parallel to $q$. Observe that $G(x)=\nabla H_S/\|\nabla H_S\|$ on $S$. Since there are at least three critical points, there are two critical points $p_0,p_1$ such that $f_q(p_0)\neq f_q(p_1)$ and $\nabla H_S(p_0)=c\nabla H_S(p_1)=c'q$ for some $c>0$ and $c'\neq0$. It remains to show that $c\neq 1$.
    
    Let $R=p\partial_p$ be the radial vector field. As $H_S$ is two-homogeneous, it satisfies $dH_s(R)=2H$. Writing this in the form of gradients, for every $p\in S$ we have $\langle p,\nabla H_S(p)\rangle =1/2$. Although $q$ is not parallel to the gradient along $S$ everywhere, for $p_0,p_1$ we also have the relation
    $f_q(p_i)=\langle p_i,q\rangle$. Since the critical values at the points $p_i$ are distinct, it now follows that $c\neq 1$.
\end{proof}

Before we can prove Theorem \ref{thm:morseofnonconvex}, we have to make a detour. For some while, $S$ will be merely a compact subset without any manifold structure.

\begin{definition}
        Let $S\subset \mathbb{R}^{k+1}$ be a compact subset. Let $F:\mathbb{R}^{k+1}\to \R$ be a continuous function. A local maximum of $F$ in $S$ is a point $x\in S$ with the following property: there exists an open set $U\subset S$ containing $x$ such that $\sup_{y\in U}F(y)=F(x)$ and $F(\partial U)$ is contained in $(-\infty,F(x))$. Here the boundary is with respect to the topology on $S$.    Denote by $$\mathrm{maxi}(F,S)$$ the set of local maximizers of $F$ in $S$.  
    \end{definition}
Observe that for a continuous function $F:\mathbb{R}^{k+1}\to \R$ and compact subset $S\subset \mathbb{R}^{k+1}$, the set $\mathrm{maxi}(F,S)$ is not necessarily closed.  However, it does contain global maxima, so it is non-empty when $S$ is non-empty. The neighbourhood $U$ in the definition of a local maximum above will be called an \textit{isolating neighbourhood} while the pair $(x,U)$ is called an \textit{isolating pair} (here we are always assuming it is clear what the $F$ and $S$ are). The image $F(\mathrm{maxi}(F,S))$ will be called the \textit{spectrum of local maximizers} (or just spectrum when it is clear what we are referring to).

\begin{lemma}\label{lem:spectrumnowheredense}
    Let $S\subset \mathbb{R}^{k+1}$ be a compact subset. Let $F:\mathbb{R}^{k+1}\to \R$ be a continuous function. The spectrum of local maximizers is no-where dense in $\R$. 
\end{lemma}
Recall that a set $A\subset \R$ is no-where dense if its closure has empty interior. 
\begin{proof}[Proof of the Lemma]
    Suppose for contradiction that there is an open non-empty subset $U=(a,b)\subset \R$ in which the spectrum of local maximizers is dense, where $a<b$ are two real numbers and $(a,b)$ is an open interval. To obtain the contradiction we will construct an open cover of $S$ without finite sub-cover.

    Let $x$ be a local maximizer with value $F(x)=c\in(a,b)$, and let $(x,U)$ be an isolating pair. For every $\epsilon>0$, $U\cap F^{-1}(c-\epsilon,c]$ will still be an isolating neighbourhood of $x$. In particular, for a given $\epsilon>0$, the set $U\cap F^{-1}(c-\epsilon/2,c]$ is open, and its closure is contained in the open set $U\cap F^{-1}(c-\epsilon,\epsilon]$.

    By assumption, the spectrum of local maximizers is dense in $(a,b)$. We can select a sequence of local maximizers $\{x_n:n\in\{1,2,3,\ldots\}\}$ such that the interval
    $$\left(F(x_n)-|a-b|/2^{n+1},F(x_n)+|a-b|/2^{n+1} \right)$$
    does not contain $F(x_{n'})$ for any integer $n'\neq n$. Let $U_n\subset S$ be an isolating neighbourhood for $x_n$. By the previous paragraph, $V_n=U_n\cap F^{-1}(F(x_n)-|a-b|/2^{n+1},F(x_n)]$ is an isolating neibhourhood of $x_n$. Let $V$ be the complement
    $$V:=S\setminus\bigcup_{n\in \N}\mathrm{cl }\:U_n\cap F^{-1}(F(x_n)-|a-b|/2^{n+2},F(x_n)],$$
    where $\mathrm{cl} \:A$ denotes the closure of a set $A\subset S$. The collection of sets $\{V,V_n:n\in \N\}$ is an open covering of $S$ which cannot have a finite sub-cover. This contradiction completes the proof.
\end{proof}

When $S$ is not convex, we want to show that $\mathrm{maxi}(F,S)$ is not connected when $F$ is one of the height functions $F_q=\langle q,\cdot\rangle$, for some $q\in\mathbb{S}^k$. To emphasize; we want to show that there is at least one height function for which this is true. The above Lemma already gives a piece of the proof:   If the spectrum of local maximizers contains more than one point, it is automatically disconnected. Hence, $\mathrm{maxi}(F,S)$ cannot be disconnected. So in the proofs below we will assume without loss of generality that $\mathrm{maxi}(F,S)$ contains only global maxima.
We first show something weaker:

\begin{lemma}
    Let $S\subset \mathbb{R}^{k+1}$ be a compact subset. Let $F_q$ denote the projection along $q\in\mathbb{S}^{k}$. If $\mathrm{maxi}(F_q,S)$ is a convex subset of $\mathbb{R}^{k+1}$ for every $q\in\mathbb{S}^{k}$, then $S$ is itself convex.
\end{lemma}
\begin{proof}
    Let $\mathrm{conv}(S)\subset\mathbb{R}^{k+1}$ denote the convex hull of $S$, which is itself a compact convex set. It is straightforward to verify that $\mathrm{maxi}(F_q,\mathrm{conv}(S))$ is convex for every $q\in \mathbb{S}^k$, and contains only global maximizers. We claim that  for every $q\in\mathbb{S}^k$ we have the equality of sets $\mathrm{maxi}(F_q,\mathrm{conv}(S))=\mathrm{maxi}(F_q,S)$. Assuming the claim is true, we are done: the equality of sets shows that the convex hull of $S$ is equal to $S$.

    Before we show the inclusions, observe that due to the convexity of $\mathrm{maxi}(F_q,S)$ and linearity of $F_q$, the set $\mathrm{maxi}(F_q,S)$ contains only global maxima of $F$ on $S$.

    $(\subset):$ let $q\in \mathbb{S}^k$. For any point $z$ in the convex hull of $S$, there exist $x,y\in S$ and a real number $t\in[0,1]$ such that $z=(1-t)x+ty$ (operations in $\mathbb{R}^{k+1}$). It follows that $\mathrm{maxi}(F_q,\mathrm{conv}(S))$ is contained in $\mathrm{maxi}(F_q,S)$. Indeed: $F_q$ restricted to $\mathrm{conv}(S)$ must have the same maximal value as $F_q$ restricted to $S$. Moreover, any point in $\mathrm{maxi}(F,\mathrm{conv}(S))$ is a convex combination of points in $\mathrm{maxi}(F,S)$, which was convex by assumption. 
    
    $(\supset):$ the inclusion of sets $\mathrm{maxi}(F,S)\subset \mathrm{maxi}(F,\mathrm{conv}(S))$ is even easier: We know that $F$ attains the same maximum value on $S$ and $\mathrm{conv}(S)$, and since $S\subset \mathrm{conv}(S)$ by definition, the inclusion follows. This proves the claim, and finishes the proof.
\end{proof}

In other words, if $S$ is not convex, there must be a height function whose set of local maximizers is also not convex. 

Given a compact subset $S$, local maxima of the linear maps $F_q$ enjoy the following stability property: If $x\in \mathrm{maxi}(F_q,S)$ be a local maximum, then there is an open neighbourhood $U\subset S$ of $x$ and an open neighbourhood $V\subset \mathbb{S}^k$ of $q$ with the following property. For every $q'\in V$ there exists a local maximum $y\in U$ of $F_{q'}$. This is straightforward to prove: $x$ is a local maximizer, so there exists a neighbourhood $U$ of $x$ in $S$ such that $F_q|_{\partial U}$ takes values less than $F_q(x)$. Since $\partial U$ is compact, there exists $\epsilon>0$ such that $F_q(x)-F_q(y)\geq \epsilon$ for all $y\in \partial U$. The function $\max_U:\mathbb{S}^k\to \R,\:q\mapsto \mathrm{max}_{y\in \overline
U}F_q(y)$ is well-defined and continuous. Hence, for $q'$ close enough to $q$ (depending on $\epsilon)$, it follows that there is a point $y(q')$ satisfying $F_{q'}(y(q'))=\mathrm{max}(q')$, while simultaneously $|\mathrm{max}(q')-\mathrm{max}(q)|<\epsilon/2$, $|F_q(z)-F_{q'}(z)|<\epsilon/2$ for all $z\in \partial U$. In particular, $y(q')$ is a local maximum. In other words:

\begin{proposition}
    Let $q\in \mathbb{S}^k$ be a unit vector, let $S\subset \mathbb{R}^{k+1}$ be a compact subset, let $F_q$ be the continuous projection along $q$ and let $x\in \mathrm{maxi}(F_q,S)$ be a local maximizer. If $U$ is a neighbourhood of $x$ isolating $x$, then there is an open neighbourhood $V\subset \mathbb{S}^k$ containing $q$ such that for each $q'\in V$, $(y(q'),U)$ is an isolating pair. 
\end{proposition}
\begin{lemma}\label{lem:nonconvexdisjointmaxi}
    Let $S\subset \R^{k+1}$ be a compact subset. If $S$ is not convex, then there exists an open non-empty set $q\in \mathbb{S}^k$ such that the set $\mathrm{maxi}(F_q)$ is not connected.
\end{lemma}
\begin{proof}
    The claimed openness follows from the above proposition. So it remains to prove that the set is non-empty.

    We work by induction on $k\in \mathbb{N}$. For $k=1$, let $S\subset \R^{2}$ be a compact subset which is not convex. By the previous Lemma, there exists $q\in\mathbb{S}^1$ such that $\mathrm{maxi}(F_q,S)$ is not convex. If $\mathrm{maxi}(F_q,S)$ contains anything other than global maxima, it follows that $\mathrm{maxi}(F_q,S)$ must be disconnected, and we are done (see Lemma \ref{lem:spectrumnowheredense} and the subsequent remark). So let us assume $\mathrm{maxi}(F_q,S)$ contains only global maximizers. Let $a\in \R$ be the maximal value of $F_q$ on $S$. Let $L:F^{-1}(a)\to \R$ be any affine isomorphism. Embed $\mathrm{maxi}(F_q,S)$ via $L$ into $\R$. The image is a non-convex compact subset of $\R$, which must therefore be disconnected. This shows that $\mathrm{maxi}(F_q,S)$ is itself disconnected. This finishes the base step.

    Assume now that the Lemma holds for $k\in \mathbb{N}$. Let $S\subset \mathbb{R}^{k+2}$ be a compact subset that is not convex. By the previous Lemma, there exists $q\in \mathbb{S}^{k+1}$ for which $\mathrm{maxi}(F_q,S)$ is not convex. By the same argument as in the base step, assume without loss of generality that $\mathrm{maxi}(F_q,S)$ contains only global maximizers. Let $a\in \R$ be the maximal value of $F_q$ on $S$. Let $L:F^{-1}(a)\to \mathbb{R}^{k+1}$ be any affine isomorphism. Embed $\mathrm{maxi}(F_q,S)$ via $L$ into $\R^{k+1}$. The image is a non-convex compact subset. By the induction hypothesis, there exists $q''\in \mathbb{S}^k$ such that 
    $$\mathrm{maxi}\Big(F_{q''},L(\mathrm{maxi}(F_q,S))\Big)\subset \mathbb{R}^{k+1}$$
    is a disconnected subset. Let $V_1,V_2\subset \mathbb{R}^{k+2}$ be two disjoint non-empty closed subsets whose union is $L^{-1}\mathrm{maxi}(F_{q''},L(\mathrm{maxi}(F_q,S)))$, which must then exist. Denote be $q'\in \mathbb{S}^{k+1}$ the preimage of $q''$ under $L$.

    For $\epsilon>0$ define the unit vector $q_{\epsilon}=(q+\epsilon q')/\|q+\epsilon q'\|$, and let $F_{\epsilon}:=F_{q_{\epsilon}}$ be the associated projection. For $\epsilon >0$, let $b(\epsilon)\in \R$ be the value such that $V_1\cup V_2\subset F^{-1}(b(\epsilon))$. Let $\epsilon_n> 0$ be a monotone sequence converging to zero. The monotonicity of the sequence implies that $F_{\epsilon_n}^{-1}[b(\epsilon_n),\infty]\cap S$ is a decreasing nested sequence of sets, i.e. 
    $$F_{\epsilon_n}^{-1}[b(\epsilon_n),\infty)\,\cap\ S\ \ \subset\ \ F_{\epsilon_{n'}}^{-1}[b(\epsilon_{n'}),\infty)\,\cap\ S\quad \forall n'\leq n.$$
    We argue by contradiction that there is an $n\in \N$ for which $F_{\epsilon_n}^{-1}[b(\epsilon_n),\infty)\cap S$ is disconnected. If all of them are connected, the countable intersection property for nested compact sets says that their countable intersection is also connected. However, this is equal to $V_1\cup V_2$, which is disconnected. 

    So, let $n$ be our distinguished integer for which $F_{\epsilon_n}[b(\epsilon_n),\infty)\cap S$ is disconnected. Let $W_1,W_2$ be two closed non-empty disjoint sets whose union is this preimage. Taking maxima over the $W_i$, it is now easy to show that $F_{\epsilon_n}$ has at least two local maxima, one in $V_1$ and one in $V_2$. This finishes the induction step, and hence the proof.
    \end{proof}

    \begin{proof}[Proof of Theorem \ref{thm:morseofnonconvex}]
    If we can prove that there is a Morse function with at least three critical points3, then it follows that there are at least three distinct values: any Morse function must have a minimum and a maximum as critical values, by compactness of $S$. If the third critical value is also a global minimum or maximum and there are no other critical values, then the Morse homology of this Morse function would not be that of a sphere, which would be a contradiction.

    Lemma \ref{lem:nonconvexdisjointmaxi} states that there is an open, non-empty set $U\subset \mathbb{S}^{k}$ for which local maximizers of $f_q$ form a disconnected set. By Proposition \ref{prop:densemorseheight} there is an open dense (hence non-empty) subset $U'$ of $U$ such that for all $q\in U'$, $f_q$ is additionally Morse. Using Morse charts it is straightforward to show that local maximizers of $f_q$ in $S$ are necessarily critical values $f_q$ with maximal index. In conclusion, these $f_q$ are Morse and have at least two global maxima, and of course a local minimum. This finishes the proof.
\end{proof}

\subsection{Homotopic short chords with distinct action}
In this final section we prove Theorem \ref{thm:nonconvexunequalchord} using Theorem \ref{thm:morseofnonconvex}, or rather Corollary \ref{cor:morseofnonconvex} from the previous section. To set notation and explain the set-up, consider a Riemannian metric $g$ on a manifold $M$ (where $(M,g)$ is not necessarily closed or complete), and $x\in M$ is a point, then the following property holds:  There exists a neighbourhood $U\subset M$ of $x$ in which every two points are uniquely connected by a $g$-geodesic. These are so-called geodesically convex neighbourhoods. From the Hamiltonian perspective, this means that any two fibers of $T^*M$ over $U$ are uniquely connected by a Hamiltonian chord in an oriented sense: from $T^*_yM$ to $T^*_zM$ ($y,z\in U$) there is a unique Hamiltonian chord of the flow associated to $H=\frac{1}{2}g_*$. 

The Hamiltonian flow of $H_S$, the two-homogeneous Hamiltonian associated with $S$ can be `linearized' on $T^*U$ to give an approximation of the genuine flow. The linearized flow will be relatively easy to describe by the `shape' of the fiber $S_x$.  More precisely, for a point $(x,p)\in S_x$, and a flow line $\gamma(t)=(x(t),p(t))$, the Hamiltonian equations read as $\dot{x}(t)=\nabla_pH(x,p)$ and $\dot{p}=-\nabla_xH(x,p)$ where the gradients are taken with respect to any metric on $U$. We linearize the flow by defining $\tilde{H}:U\to \R$, $\tilde{H}(y,p)=H(x,p)$ for all $y\in U$. In other words, the $p$-coordinate is constant, and the $x$-coordinate evolves in a straight line with the direction being the direction of the fiber-gradient of $H_S$ at $x$.

\begin{theorem}\label{thm:nonconvexunequalchord}
    Let $S\subset T^*M$ be a fiber-wise star-shaped contact-type hypersurface, and $x\in M$ a point such that $S_x=T^*_xM\cap S$ does not bound a convex set. Then, for every neighbourhood $U\subset M$ of $x$, there exists $y\in U$, $y\neq x$, and Reeb chords $\gamma_i:[0,T_i]\to S$ for $i=0,1$, such that:
    \begin{enumerate}
        \item $\pi\circ\gamma_i(0)=x$ and $\pi\circ\gamma_i(T_i)=y$, $i=0,1$,
        \item the terminal times $T_i$ are not equal.
        \item the curves $\pi\circ\gamma_i$ are homotopic relative end-points.  
    \end{enumerate}
\end{theorem}

For $c\in \R$, let $m_c:\R^n\to \R^n$ be the vector space multiplication by $c$. 
\begin{lemma}\label{lem:uniformlyscaled}
    Let $H:\R^{2n}\cong T^*\R^n\to \R$ be a smooth Hamiltonian which does not depend on the base, i.e. $H(q,p)=H(q',p)$ for all $q,q'\in \R^n$. Let $G:T^*\R^n\to R$ be any other smooth Hamiltonian such that $G|T^*_0M=H|T^*_0M$ (i.e. the maps agree on the cotangent space at zero). 
    
    Let $G_c:=G\circ m_c$, i.e. $G_c(q,p)=G(m_cq,p)$. As $c\to 0$, $\phi^1_{G_c}\to \phi^1_H$ uniformly on any compact neighbourhood of $0\in T^*\R^n$.
\end{lemma}
\begin{proof}
    On any compact neighbourhood of $0\in T^*\R^n$, $G_c$ will converge uniformly to $G(0,p)=H(q,p)$. Hence, on any compact neighbourhood of zero, $\phi^1_{G_c}\to \phi^1_{G(0,p)}$ uniformly. This finishes the lemma.
\end{proof}

For a set $X$, let $\Delta X$ denote the diagonal in $X\times X$.
\begin{lemma}\label{lem:transverseselfintersect}
    Let $H_S$ be the two-homogeneous Hamiltonian associated to $S$, and let $x\in M$ be a point. Let $\tilde{H}$ be another two-homogeneous Hamiltonian such that in a chart $\phi:U\to \R^n$ around $x\in U$, $\tilde{H}\circ \phi^{-1}_*$ does not depend on the base.

    Let $\phi_H^1:T_{x}^*M\to T^*M$ and $\phi^1_{\tilde{H}}:T_{x}^*M\to T^*M$ denote the restriction of the time-one Hamiltonian flows to the fiber $T^*_xM$. If the product map $$\phi_{\tilde{H}}^1\times \phi_{\tilde{H}}^1|T^*_xM\times T^*_xM\setminus \triangle(T^*_xM)$$ intersects the diagonal $\triangle M$ transversely and non-trivially, then $\phi^1_H\times \phi^1_H$ intersects $\triangle M$ non-trivially.
\end{lemma}
\begin{proof}
    Lets work in a chart $\phi:U\to V$ containing $x$, with image $V$ in $\R^n$ and $\phi(x)=0$. For a sequence $c_n\to 0$, $m_{c^{-1}}\circ \phi$ is a chart for increasingly smaller neighbourhoods of $x$ onto $U$.

    By Lemma \ref{lem:uniformlyscaled}, with respect to this chart, the Hamiltonian flows of the rescaled Hamiltonians in $V$ are converging uniformly to the flow of $\tilde{H}$ in a neighbourhood of $(0,0)\in T^*U$. 
    
    Let $M,N$ be smooth manifolds without boundary, and let $K\subset N\times N$ be a submanifold. Let $Z\subset M\times M$  be a compact set not intersecting $\triangle M$. It is well known that the set $\{f\in C^{\infty}(M,N):f\times f\pitchfork K\mathrm{\:over\:}Z\}$ is an open set in the Whitney $C^1$ topology. 
    
    Let $(x,y)$ be a point off of the diagonal, where $\phi^1_{\tilde{H}}$ has a point of self-intersection. Let $Z$ be any compact neighbourhood of this point in $T_0^*\R^n\times T_0^*\R^n$ not intersecting the diagonal. The Hamiltonian flow $\tilde{H}$ is invariant under scaling, meaning $\phi^{c^{-1}t}_{\tilde{H}}(m_cx)$ is a constant. Transversality of self-intersections is therefore also independent of scaling. By choosing a small enough chart around $x$, we have picked a Hamiltonian flow close enough to the flow of $\tilde{H}$ in the bundle chart over $V$, so that it lies in the open set of transverse self-intersection over $Z$. 

    It remains to show that the set of self-intersections over $Z$ is not empty. The maps $m_c$ where $c$ varies on the interval $(-\epsilon,\epsilon)$ for some $\epsilon>0$ small enough, defines a path of Hamiltonians in $V$ with the Hamiltonian at zero being equal to $\tilde{H}$, and hence a parameterized family of flows converging to $\phi^1_{\tilde{H}}$ in $V$. For this smooth family of flows with domain $T^*_0M\times (-\epsilon,+\epsilon)$, the non-trivial transverse self intersection of $\phi_{\tilde{H}}^1$ at $(x,y)$ implies that for $\epsilon$ small enough, there is a $1$-dimensional family of such non-trivial self-intersections. I.e. the preimage of self-intersections is locally parameterized by $\epsilon$. This finishes the proof.
\end{proof}
\begin{proof}[Proof of the Theorem \ref{thm:nonconvexunequalchord}]
    Let $U$ be any neighbourhood of $x$ in $M$, and without loss of generality we can assume that $U$ is the domain of a chart $\phi:U\xrightarrow{\sim} V\subset \R^n$, which places us in $\R^n$. We will work in the chart induced over $T^*U\cong \R^n\times V$, where $V$ is the image of the chart with $\phi(x)=0$. That is, let $D\subset \R^n$ be the star-shaped domain which $S_x$ bounds, and let $H:\R^n\to \R$ be the two-homogeneous maps associated to it.  Extend $H$ to $\R^n\times V$ to be constant in the fibers, i.e.\ by $H(p,q)=H(p)$ for all $q\in V$.

    Hamilton's equations $\dot{q}=\nabla_p H, \dot{p}=-\nabla_qH$ imply that the flow (for small enough $t$, not leaving the chart) of $\phi_H$ is determined by the gradient of $H:\R^n\to \R$. Corollary \ref{cor:morseofnonconvex} implies that there are distinct small times $t_1,t_2>0$ and a direction $q\in \mathbb{S}^{n-1}$ and two distinct $p_i\in S_x$ such that $\pi\circ \phi^{t_1}(p_1)=\pi \circ\phi_H^{t_2}(p_2)$. Indeed: the Corollary implies there is a distinghuished $q$ for which the associated height function $f_q$ on $S_x$ two critical values, where the gradient $\nabla H$ are positively parallel but of unequal size. This means that the flow of $H$ will hit points in the direction of $q$ for positive but distinct times.
    
    By rescaling time and the norm of the $p_i$, we can instead assume that $t_1=t_2$, but the $p_i$ lie on different level sets of $H$. Moreover, since $G$ has a regular value at these points, $\pi\circ \phi_H^t$ has regular values at the image. In other words, the self-intersection is transverse.  By Lemma \ref{lem:transverseselfintersect}, this means that the original Hamiltonian flow has a transverse non-trivial self-intersection in $U$. Additionally, the uniform convergence implies that for small enough $U$, the self-intersection of the original flow also happens for points in distinct level sets of $H_S$. This finishes the proof.
\end{proof}

\bibliography{biblio}

\end{document}